\documentclass[10pt, a4paper, reqno, english]{amsart}

\usepackage{amssymb, mathrsfs}

\usepackage[draft=false,hidelinks]{hyperref}

\theoremstyle{plain}
\newtheorem*{theorem*}{Theorem}
\newtheorem{prop}{Proposition}
\newtheorem{lemma}[prop]{Lemma}
\theoremstyle{remark}
\newtheorem*{remark*}{Remark}

\numberwithin{prop}{section}
\numberwithin{equation}{section}

\newcommand\ab{\mathbf{a}}
\newcommand\db{\mathbf{d}}
\newcommand\eb{\mathbf{e}}
\newcommand\oneb{\mathbf{1}}
\newcommand\ddd{\,\mathrm{d}}
\newcommand\PP{\mathbb{P}}
\newcommand\QQ{\mathbb{Q}}
\newcommand\Qbar{{\overline{\QQ}}}
\newcommand\RR{\mathbb{R}}
\newcommand\Ab{\mathbb{A}}
\newcommand\CC{\mathbb{C}}
\newcommand\ZZ{\mathbb{Z}}
\newcommand\sums[1]{\sum_{\substack{#1}}}
\newcommand\cfrb{{\underline{\mathfrak{c}}}}
\newcommand\afr{{\mathfrak{a}}}
\newcommand\bfr{{\mathfrak{b}}}
\newcommand\cfr{{\mathfrak{c}}}
\newcommand\dfr{{\mathfrak{d}}}
\newcommand\efr{{\mathfrak{e}}}
\newcommand\ffr{{\mathfrak{f}}}

\newcommand\pfr{{\mathfrak{p}}}
\newcommand\qfr{{\mathfrak{q}}}
\newcommand\dfrb{{\underline{\mathfrak{d}}}}
\newcommand\efrb{{\underline{\mathfrak{e}}}}
\newcommand\afrb{{\underline{\mathfrak{a}}}}
\newcommand{\Cs}{\mathscr{C}}
\newcommand{\Ps}{\mathscr{P}}
\newcommand{\Ms}{\mathscr{M}}
\newcommand{\Is}{\mathscr{I}}
\newcommand{\Fs}{\mathscr{F}}
\newcommand{\Gs}{\mathscr{G}}
\newcommand{\Rs}{\mathscr{R}}
\newcommand{\Ss}{\mathscr{S}}
\newcommand{\Os}{\mathcal{O}}
\newcommand{\OK}{\Os_K}
\newcommand{\IK}{\Is_K}
\newcommand\Nt{\widetilde{N}}
\newcommand\Pt{\widetilde{P}}
\newcommand\Msunder{\underline{\Ms}}
\newcommand\Msover{\overline{\Ms}}
\newcommand\Munder{\underline{M}}
\newcommand\Mover{\overline{M}}
\newcommand\Aover{\overline{A}}
\newcommand\N{\mathfrak{N}}
\newcommand{\GG}{\mathbb{G}}
\newcommand{\GGm}{\GG_{\mathrm{m}}}
\newcommand\bigwhere[2]{\left\{#1:\ \begin{aligned}#2\end{aligned}\right\}}
\newcommand\congr[3]{#1 \equiv #2 \pmod{#3}}
\newcommand\id{\mathrm{id}}
\DeclareMathOperator{\vol}{vol}
\DeclareMathOperator{\Pic}{Pic}
\DeclareMathOperator{\Spec}{Spec}
\DeclareMathOperator{\Gal}{Gal}

\DeclareMathOperator{\Cl}{Cl}
\newcommand{\Mod}[1]{\ (\mathrm{mod}\ #1)}

\begin{document}

\title[Quintic del Pezzo surfaces over number fields]
{Points of bounded height on quintic del Pezzo surfaces over number fields}

\author{Christian Bernert}

\address{Institut f\"ur Algebra, Zahlentheorie und Diskrete Mathematik, Leibniz Universit\"at Hannover, Welfengarten 1, 30167 Hannover, Germany}

\email{bernert@math.uni-hannover.de}

\author{Ulrich Derenthal} 

\address{Institut f\"ur Algebra, Zahlentheorie und Diskrete Mathematik, Leibniz Universit\"at Hannover, Welfengarten 1, 30167 Hannover, Germany}

\address{School of Mathematics, Institute for Advanced Study, 1 Einstein Drive, Princeton, New Jersey, 08540, USA}

\email{derenthal@math.uni-hannover.de}

\date{September 24, 2025}


\setcounter{tocdepth}{1}

\maketitle

{\centering\footnotesize To Yuri Tschinkel on his 60th birthday.\par}

\begin{abstract}
  We prove Manin's conjecture for split smooth quintic del Pezzo surfaces over arbitrary number fields with respect to fairly general anticanonical height functions. After passing to universal torsors, we first show that we may restrict the torsor variables to their typical sizes, and then we can solve the counting problem in the framework of o-minimal structures.
\end{abstract}

\tableofcontents

\section{Introduction}

In his Ph.D. thesis \cite{T92}, Tschinkel proved that the number of rational points $x$ of anticanonical height $H(x) \le B$ in the complement $U$ of the ten lines on a split smooth quintic del Pezzo surface $X$ over an arbitrary number field $K$ is $O(B^{1+\epsilon})$. In this article, we prove Manin's conjecture \cite{FMT89} for such $X$ over $K$, namely the asymptotic formula
\begin{equation*}
    N_{U,H}(B):=|\{x \in U(K) : H(x) \le B\}| = c_{X,H} B(\log B)^4(1+o(1))
\end{equation*}
as $B \to \infty$, where $c_{X,H}$ is Peyre's constant \cite{Pey95}.

Previously, Manin's conjecture for split smooth quintic del Pezzo surfaces was known only over $\QQ$. Here, Manin and Tschinkel \cite[Theorem~1.9]{MT93} proved the upper bound $N_{U,H}(B) \ll B(\log B)^6$, Salberger\footnote{Lecture ``Counting rational points on del Pezzo surfaces of degree $5$'', Bern, 1993} proved $N_{U,H}(B) \ll B(\log B)^4$, and de la Bret\`eche \cite{Bre02} proved the asymptotic formula above over $\QQ$, using the universal torsor method. Browning \cite{Br22} gave a new proof of this asymptotic formula over $\QQ$ with an improved error term, using conic fibrations.

Furthermore, de la Bret\`eche and Fouvry \cite{BF04} proved Manin's conjecture for a certain nonsplit smooth quintic del Pezzo surface over $\QQ$. Recently, Heath-Brown\footnote{Lecture ``Manin’s Conjecture for Del Pezzo Surfaces of Degree 5 with a Conic Fibration'', Institut Mittag-Leffler, 2024} announced joint work with Loughran \cite{HBL25} on Manin's conjecture for nonsplit smooth quintic del Pezzo surfaces over $\QQ$ with a conic fibration in the generic case.

Over number fields other than $\QQ$, much less is known. Glas and Hochfilzer \cite[Theorem~1.3]{GH24} recently generalized Tschinkel's bound $O(B^{1+\epsilon})$ to all smooth quintic del Pezzo surfaces with a conic bundle structure.

Our works seems to be the first instance where Manin's conjecture is established for a smooth del Pezzo surface of degree less than $6$ over a number field $K$ different from $\QQ$ (note that when the degree is at least $6$, smooth del Pezzo surfaces are toric, so that Manin's conjecture is known over all number fields by the work of Batyrev and Tschinkel \cite{BT98}).

In fact, even for $K=\mathbb{Q}$, our work both simplifies the argument and generalizes the result from \cite{Bre02} since we are able to consider more general anticanonical height functions than just the most symmetric choice used in \cite{Bre02} and \cite{Br22}, as we describe now.

\subsection{Height functions}\label{sec:heights}

Since we assume that our quintic del Pezzo surface $X$ over a number field $K$ is split (i.e., each of its ten lines is defined over $K$), it is isomorphic to a blow-up $\pi: X \to \PP^2_K$ of the projective plane in
\begin{equation*}
  p_1=(1:0:0),\ p_2=(0:1:0),\ p_3=(0:0:1),\ p_4=(1:1:1).
\end{equation*}
Let $V$ be the complement of the six lines through two of $p_1,\dots,p_4$ in $\PP^2_K$, which is isomorphic to the complement $U$ of the ten lines on $X$.

In the definition of the height functions and the statement of our main result, we use standard notation for the invariants of the number field $K$ (see Section~\ref{sec:notation}).

We will work with the following natural class of anticanonical height functions: Consider the six-dimensional vector space of polynomials of anticanonical degree (i.e. homogeneous of degree $3$) in $K[Y_1,Y_2,Y_3]$ vanishing in $p_1,p_2,p_3,p_4$. Suppose that $\Ps$ is a finite generating set of this vector space that consists of polynomials with coefficients in $\OK$ satisfying
\begin{equation} \label{eq:height_gcd}
    \gcd_{P \in \Ps} P(y)=\frac{\gcd(y_2,y_3)\gcd(y_1,y_3)\gcd(y_1,y_2)\gcd(y_1-y_2,y_1-y_3)}{\gcd(y_1,y_2,y_3)}
\end{equation}
(as an equality of ideals) for all triples $y=(y_1:y_2:y_3) \in \PP^2_K(K) \setminus \{p_1,p_2,p_3,p_4\}$. Note that the four factors on the right-hand side correspond to the blown-up points $p_i$.

For $y \in V(K)$, we define
\begin{equation}\label{eq:height_X}
  H_0(y):= \prod_{v \in \Omega_K} \max_{P \in \Ps} |P(y)|_v.
\end{equation}
Let $H$ be the height function on $X$ induced by $H_0$, with $H(x):=H_0(\pi(x))$ for $x \in U(K)$. We say that $H$ is an \emph{admissible anticanonical height function} if $\Ps$ has the properties above.

\begin{remark*}
    For $\sigma$ in the symmetric group $S_3$, we define the cubic polynomials
    \begin{align*}
        P_\sigma&:=Y_{\sigma(1)}Y_{\sigma(2)}(Y_{\sigma(1)}-Y_{\sigma(3)}),\\
        Q_\sigma&:=Y_{\sigma(2)}(Y_{\sigma(1)}-Y_{\sigma(2)})(Y_{\sigma(1)}-Y_{\sigma(3)}).
    \end{align*}
    Admissible anticanonical height functions are given, for example, by $\Ps=\{P_{\sigma}: \sigma \in S_3\}$ (which corresponds to the standard Weil height with respect to the most natural anticanonical embedding of $X$ into $\PP^5_K$, defined by this linear system $\Ps$) or $\Ps=\{P_{\sigma},Q_{\sigma}: \sigma \in S_3\}$ (which leads to the most symmetric height on the universal torsor).

    In principle, it would also be possible to treat height functions coming from sets $\Ps$ which do not satisfy the coprimality condition \eqref{eq:height_gcd}. However, this would require changing finitely many of the local densities. We have chosen to restrict to the class of height functions described above, in the interest of notational simplicity.

    Note that our approach allows us to treat quite general height functions, while the most symmetric choice $\Ps=\{P_{\sigma},Q_{\sigma}: \sigma \in S_3\}$ is crucial for \cite{Bre02} (see the discussion around \cite[(1.3)]{Bre02}) and apparently also for \cite{Br22}.
\end{remark*}

\subsection{The main result}\label{sec:main_result}

\begin{theorem*}
    Let $H$ be an admissible anticanonical height function (as in Section~\ref{sec:heights}).
    For $B \ge 3$, we have
    \begin{equation*}       
        N_{U,H}(B) = c_{X,H}B(\log B)^4 +O\left(\frac{B(\log B)^4}{(\log\log B)^{\frac{1}{3d+1}}}\right),
    \end{equation*}
    where
    \begin{equation*}
        c_{X,H} = \alpha(X)\cdot \left(\frac{2^{r_1}(2\pi)^{r_2}R_Kh_K}{|\mu_K|\cdot |\Delta_K|^{\frac 1 2}}\right)^5 \cdot \frac{1}{|\Delta_K|} \prod_{v \in \Omega_K} \omega_v(X)
    \end{equation*}
    is the constant predicted by Peyre, with
    \begin{equation*}
        \alpha(X) = \frac{1}{144}
    \end{equation*}
    and
    \begin{equation*}
        \omega_v(X) = \begin{cases}
            \left(1-\frac{1}{\N\pfr}\right)^5\left(1+\frac{5}{\N\pfr}+\frac{1}{\N\pfr^2}\right),& \text{$v = \pfr$ prime,}\\
            \frac{3}{2}\vol\{y \in \RR^3 : \max_{P \in \Ps}|P(y)|_v \le 1\}, & \text{$v$ real,}\\
            \frac{12}{\pi} \vol\{y \in \CC^3 : \max_{P \in \Ps}|P(y)|_v\le 1\}, & \text{$v$ complex.}
        \end{cases}
    \end{equation*}
\end{theorem*}

\subsection{Overview of the proof}

We use the universal torsor method. In fact, Salberger's and de la Bret\`eche's \cite{Bre02} work on quintic del Pezzo surfaces over $\QQ$ were probably the first applications of the universal torsor method to Manin's conjecture, besides Salberger's work on toric varieties over $\QQ$ \cite{Sal98}. In the meantime, this method has been applied to many singular del Pezzo surfaces (see the references in \cite{D14}, for example) and some higher-dimensional Fano varieties (see \cite{Bre07,BBS14,BBDG24}, for example); in almost all these cases, the universal torsor is given by a single equation. In a series of papers including \cite{DF14,FP16}, the universal torsor method in its most basic form (as described in \cite{D09}) has been generalized to number fields beyond $\QQ$ and applied over arbitrary number fields to the possibly easiest nontoric example of singular quartic del Pezzo surfaces of type $A_1+A_3$ \cite{FP16,DP20}.

The first step in the universal torsor method, namely the parameterization of rational points on the given variety by integral points on its universal torsors, is now well-understood, at least for split Fano varieties over $\QQ$ and, by the work of Frei and Pieropan \cite[\S 2--4]{FP16}, also over arbitrary number fields. Hence we will be very brief in our passage to (the $h_K^5$ twists of an integral model of) a universal torsor in Proposition~\ref{prop:torsor_parameterization}, which generalizes the parameterization over $\QQ$ described in \cite[\S 1.4, \S 2.3]{Bre02}. In particular, our torsor equations are again the five Pl\"ucker equations (see \eqref{eq:torsor} below) of the Grassmannian $G(2,5)$ of two-dimensional subspaces of $K^5$, together with certain coprimality and height conditions on the ten variables $a_i, a_{jk}$ (with $i$ and $j<k$ in $\{1,\dots,4\}$).

Counting the integral points on the universal torsors would be relatively straightforward over imaginary quadratic fields, even when the class number $h_K$ is greater than $1$, since the techniques used by \cite{Bre02} are compatible with the techniques developed in this setting in \cite{DF14}.

However, we encountered a fundamental problem when trying to generalize the counting strategy from \cite{Bre02} to arbitrary number fields, using the techniques from \cite[\S 5--11]{FP16}. In the following, we sketch the the basic strategy together with this problem and our solution.

The torsor equations \eqref{eq:torsor} allow us to eliminate three of the variables, which we have chosen to be $a_{13}, a_{14}, a_{24}$. This puts certain congruence conditions on the remaining variables. Now the basic strategy is to sum over the remaining variables subject to these congruence conditions, as well as the height and coprimality conditions, repeatedly replacing sums by integrals (or lattice point counts by volumes). This turns out to work nicely if the variables $a_{ij}$ are restricted to their typical sizes, given by a function $B_{ij}$ depending on $B$ and the $a_i$ (see \eqref{eq:Bij}). In general, however, such a restriction is not guaranteed.

In \cite{Bre02}, this issue is dealt with as follows: It is not hard to see that always \textit{some} of the $a_{ij}$ are bounded by $B_{ij}$. One can thus partition the set of values of $a_i$ into sets $E_i$, according to which of the variables $a_{ij}$ are small. Then \cite{Bre02} proceeds by adapting the order of summation on each set $E_i$, allowing us to sum over the larger variables first. This leads to several complications of the entire argument already over $\mathbb{Q}$, one of them being that the $E_i$ do not behave in a completely symmetric way.

More importantly for us, it seems impossible to implement this strategy over general number fields. The main issue here is that such a partition exists for each archimedean valuation independently, but it is not possible to adapt the order of summation for each valuation separately.

One of the crucial new ingredients in our proof is therefore Proposition~\ref{prop:E_WK}, showing that there is no loss of generality in restricting to the solutions with $|a_{ij}|<WB_{ij}$ where $W$ is a slowly growing function of $B$, thus achieving an essentially optimal truncation of their range. This simplifies and streamlines the argument already over $\QQ$ and makes a generalization of the aforementioned basic strategy to number fields feasible. We hope that this idea generalizes to other instances of Manin's conjecture.

Further differences to \cite{Bre02} are: While de la Bret\`eche goes back and forth several times between the counting problem on the universal torsor and the counting problem on the quintic del Pezzo surface, our counting argument is done exclusively on the universal torsor, which seems more natural and transparent to us. Furthermore, instead of reducing by the full $S_5$-symmetry in the first steps \cite[\S 2]{Bre02}, we only use a $\ZZ/5\ZZ$-symmetry in Lemma~\ref{lem:weyl_group_symmetry}, which has the advantage that the resulting counting problem still has an $S_4$-symmetry. This also allows us to treat more general height functions.

\subsection{Plan of the article}

While describing the structure of our article, we mention some of the other key elements of our proof.

In Section~\ref{sec:conjecture}, we show that the leading constant $c_{X,H}$ in our main term agrees with the one predicted by Peyre \cite{Pey95}.

In Section~\ref{sec:parameterization}, we set up the counting problem on the universal torsor and express the dependent variables via congruences. For the generalization to number fields, we also restrict our variables to a suitable fundamental domain for the action of the unit group. An important catalyst in the process of singling out the variables $a_i$ for special treatment in the initial stages of the argument is a symmetry condition which we first introduce in Lemma \ref{lem:weyl_group_symmetry} and later remove again in Proposition~\ref{prop:remove_symmetry}.

In Section~\ref{sec:restrictions}, we perform some preliminary maneuvers that are necessary for the execution of the counting strategy outlined above. This includes the aforementioned restriction of the range of the $a_{ij}$ (Proposition \ref{prop:E_WK}), but also a truncation of the range of the $a_i$ (Section~\ref{sec:restrict_a'}).

In Section~\ref{sec:main_contribution}, we then initiate the main counting argument, first incorporating the coprimality conditions by Möbius inversion, thus reducing our problem to a lattice point count in a bounded region (Proposition~\ref{prop:congruences_as_lattice}). As in \cite[\S 9--10]{FP16}, this is then tackled by a suitable version of the ``Lipschitz principle'', as supplied by a general result of \cite{BW14} in the context of o-minimal structures \cite{Wilkie96}.

To sum the error terms arising in this process, we first need to truncate the size of the Möbius variables. In the course of this argument, a crucial ingredient is an essentially optimal upper bound for the number of solutions without the coprimality conditions, given by Lemma~\ref{lem:upperbound}. Over $\mathbb{Q}$, \cite{Bre02} reverts to an argument of Manin--Tschinkel \cite{MT93}, but their method does not seem to generalize to number fields. Instead, we rebuild our initial argument leading to the required upper bound, taking into account the missing coprimality conditions.

Finally, the sum over the remaining variables $a_1,a_2,a_3,a_4$ is carried out in Lemma \ref{lem:remain_sum}, using the general machinery of \cite{DF14}.

\subsection{Notation and conventions}\label{sec:notation}

For our number field $K$ with $r_1$ real embeddings, $r_2$ pairs of complex embeddings, class number $h_K$, regulator $R_K$, and discriminant $\Delta_K$, let $\mu_K$ be the group of roots of unity. By the analytic class number formula, the residue of its Dedekind zeta function $\zeta_K(s)$ at $s=1$ is 
\begin{equation*}
    \rho_K:=\frac{2^{r_1}(2\pi)^{r_2}R_K h_K}{|\mu_K|\cdot|\Delta_K|^{\frac 1 2}}.
\end{equation*}

Let $d:=[K:\QQ]$ be the degree, $\OK$ the ring of integers, and $U_K$ the subgroup of its unit group $\OK^\times$ generated by a chosen system of fundamental units, so that $\OK^\times = U_K \times \mu_K$. Let $\IK$ be the monoid of nonzero ideals in $\OK$. For $\afr \in \IK$, let $\mu_K(\afr)$ be the M\"obius function. For a fractional ideal $\qfr$ of $K$ and $a,b \in K$, we write $\congr{a}{b}{\qfr}$ if $a-b \in \qfr$, and $\N\qfr$ for its absolute norm. Let $\Cl_K$ be the ideal class group of $K$. Let $\Cs$ be a set of $5$-tuples $(\cfr_0,\dots,\cfr_4)$ of fractional ideals that are a system of representatives for $\Cl_K^5$.

Let $\Omega_K$ be the set of places of $K$, with the subsets $\Omega_\infty$ of archimedean places and $\Omega_f$ of nonarchimedean places; we also write $v \mid \infty$ for $v \in \Omega_\infty$. For $v \in \Omega_K$ lying above $w \in \Omega_\QQ$, let $\sigma_v : K \to K_v$ be the embedding of $K$ into its completion $K_v$ with respect to $v$, let $d_v:=[K_v:\QQ_w]$ the the local degree, and let $|\cdot|_v = |N_{K_v/\QQ_w}(\cdot)|_w$ (where $|\cdot|_w$ is the usual $p$-adic or real absolute value on $\QQ_w$). For $a \in K$, we also write $a^{(v)}$ for $\sigma_v(a)$, and $|a|_v$ for $|\sigma_v(a)|_v$; we write $N(a) = N_{K/\QQ}(a)$ for its norm over $\QQ$. Let $\sigma=(\sigma_v)_{v\mid\infty} : K \to \prod_{v\mid\infty} K_v$; we use the same notation for its coordinate-wise extension on $K^n$.

The letter $\pfr$ always denotes a nonzero prime ideal in $\OK$, corresponding to some $v \in \Omega_f$, with $\pfr$-adic valuation $v_\pfr$ on $K_\pfr:=K_v$. Products over $\pfr$ run through nonzero prime ideals of $\OK$, possibly subject to further conditions as stated in each case. Given $a_1,\dots,a_n \in \OK$, we write $\gcd(a_1,\dots,a_n)$ for the ideal $a_1\OK+\dots+a_n\OK$.

When we use Vinogradov's $\ll$-notation or Landau's $O$-notation, the corresponding inequalities are meant to hold for all values in the relevant range, and the implied constants may depend only on $K$ and on the choice of $\Ps$ defining the height function. We write $X_1 \asymp X_2$ for $X_1 \ll X_2 \ll X_1$.

Volumes of subsets of $\RR^n$ or $\CC^n \cong \RR^{2n}$ are computed with respect to the usual Lebesgue measure, unless stated otherwise.

For our height bound $B$, we assume $B \ge 3$. We will use parameters $T_1,T_2,W$ that are suitable functions of $B$. More precisely, starting in Section~\ref{sec:restrict_a'}, we will fix $W=(\log\log B)^{1/(3d+1)}$ and $T_i=\exp(c_i \log B/\log\log B)$ for $i=1,2$ with constants $c_1,c_2>0$ such that $c_1$ is sufficiently large and $c_2>24c_1$.

We will repeatedly use the fact that the number $\tau_K(\afr)$ of ideals dividing a given ideal $\afr$ of norm $O(B)$ can be bounded by
\begin{equation}\label{eq:divisor_bound}
    \exp\left(c\frac{\log B}{\log\log B}\right) = T_1^{\frac{c}{c_1}},
\end{equation}
where the exponent here can be made arbitrarily small if we choose $c_1$ sufficiently large. This standard bound follows from the familiar bound over $\mathbb{Q}$ after taking norms.

The indices $i,j,k,l$ (or a subset of them) are pairwise distinct elements of $\{1,2,3,4\}$, and all statements involving such indices are meant for all possible values (unless something else is stated or clear from the context). We will encounter variables such as $a_{ij}$ with double indices $i\ne j$; here and in similar cases, we use the convention $a_{ji}=a_{ij}$.

\subsection*{Acknowledgements}

Most of this work was done while the second author was a Member of the Institute for Advanced Study for the academic year 2023/2024. Its hospitality and its support through the Charles Simonyi Endowment is gratefully acknowledged. We thank the anonymous referee for reading our article carefully and suggesting a number of improvements.

\section{The Manin--Peyre conjecture}\label{sec:conjecture}

By \cite{FMT89,Pey95,Pey03}, we expect
\begin{equation*}
    N_{U,H}(B) = c_{X,H}B(\log B)^{\rho-1}(1+o(1)),
\end{equation*}
as $B \to \infty$, where $\rho$ is the rank of $\Pic(X)$, and
\begin{equation*}
    c_{X,H} = \alpha(X)\beta(X)\tau_H(X).
\end{equation*}
In this section, we discuss and compute the constants appearing in this conjectural formula for our smooth quintic del Pezzo surface $X$ with respect to our anticanonical height functions $H$. This will show that the asymptotic formula in our Theorem (as stated in Section~\ref{sec:main_result}) agrees with the Manin--Peyre conjecture.

Since $X$ is a blow-up of $\PP^2_K$ in four rational points, we have $\rho=5$, and $\Gal(\Qbar/K)$ acts trivially on the geometric Picard group $\Pic(X_\Qbar)$ (where $X_\Qbar$ is the base change of $X$ to an algebraic closure $\Qbar$ of $K$). Therefore, we have
\begin{equation*}
    \beta(X) := |H^1(\Gal(\Qbar/K), \Pic(X_\Qbar))| = 1.
\end{equation*}
By \cite[\S 1.3]{Bre02} or \cite[Theorem~4]{D07}, we have
\begin{equation}\label{eq:alpha}
    \alpha(X)=\frac{1}{144}.
\end{equation}
As in \cite{Pey95, Pey03} and \cite[\S 13]{FP16}, we obtain an adelic metric on the anticanonical line bundle $\omega_X^{-1}$ that induces our height function \eqref{eq:height_X} and also local measures $\omega_{H,v}$ and a Tamagawa measure $\tau_H$. Using notation from \cite[Notation~4.5]{Pey03}, we have
\begin{equation*}
    \tau_H(X) = \lim_{s\to 1}(s-1)^5 L_S(s,\Pic(X_\Qbar))\cdot\frac{1}{|\Delta_K|}\cdot \prod_{v \in \Omega_K} \lambda_v^{-1}\omega_{H,v}(X(K_v))
\end{equation*}
for a finite set $S$ of finite places of $K$. For any place $v$, the $v$-adic density is
\begin{equation}\label{eq:v-adic_density}
    \omega_{H,v}(X(K_v)) = \iint_{K_v^2} \frac{\ddd x_2 \ddd x_3}{\max_{P \in \Ps}|P(1,x_2,x_3)|_v}
\end{equation}
with the usual Lebesgue measure on $K_v=\RR$, but twice the usual Lebesgue measure on $K_v=\CC$.

\begin{lemma}\label{lem:expected_archimedean_densities}
    For every $v \mid \infty$, the constant $\omega_v(X)$ in our Theorem agrees with $\omega_{H,v}(X(K_v))$.
\end{lemma}

\begin{proof}
    For real $v$, see \cite[\S 1.3]{Bre02}. For complex $v$, we compute (using the usual Lebesgue measure on $\RR$ and $\CC \cong \RR^2$)
\begin{align*}
    &\frac{12}{\pi} \vol\{y = (y_1,y_2,y_3) \in \CC^3 : \max_{P \in \Ps}|P(y)|_v\le 1\}\\
    &=\frac{12}{\pi} \int_{\max_{P \in \Ps}|P(y)|_v\le 1} \ddd y_1 \ddd y_2 \ddd y_3\\
    &=\frac{12}{\pi} \int_{|u_1|^3_v \max_{P \in \Ps}|P(1,u_2,u_3)|_v\le 1} |u_1|_v^2\ddd u_1 \ddd u_2 \ddd u_3\\
    &=12 \int_{t^3 \max_{P \in \Ps}|P(1,u_2,u_3)|_v\le 1,\ t \in \RR_{>0}} t^2\ddd t \ddd u_2 \ddd u_3\\
    &=4 \int_{\CC^2} \frac{1}{\max_{P \in \Ps}|P(1,u_2,u_3)|_v} \ddd u_2 \ddd u_3\\
    &=\omega_{H,v}(X(\CC)).
\end{align*}
Here, we use the complex change of variables $y_1=u_1$, $y_2=u_1u_2$, $y_3=u_1u_3$ (with Jacobi determinant $|(\partial y_i/\partial u_j)_{i,j}|_v = |u_1|_v^2$) in the second step, polar coordinates $u_1=\sqrt{t}e^{i\phi}$ with positive real $t$ and $0 \le \phi < 2\pi$ (with $\ddd u_1 = \frac 1 2 \ddd t \ddd \phi$, and where the integration over $\phi$ gives a factor $2\pi$ since the integral does not depend on the argument $\phi$ of $u_1$, but only on $|u_1|_v = t$) in the third step, $\int_{0 < t \le 1/M^{1/3}} t^2 \ddd t = 1/(3M)$ for any nonnegative $M$ in the fourth step, and \eqref{eq:v-adic_density} (with twice the usual Lebesgue measure on $\CC$) in the fifth step.
\end{proof}

\begin{lemma}\label{lem:expected_p-adic_densities}
    For every finite place $v=\pfr$ of $K$, we have
    \begin{equation*}
        \left(1-\frac{1}{\N\pfr}\right)^5\omega_{H,\pfr}(X(K_\pfr)) = \omega_\pfr(X),
    \end{equation*}
    with $\omega_\pfr(X)$ as in our Theorem.
\end{lemma}

\begin{proof}
    We compute \eqref{eq:v-adic_density} for $v = \pfr$. Let $q := \N\pfr$, $\alpha=v_\pfr(x_2)$, $\beta=v_\pfr(x_3)$. Let $\Os_\pfr^\times = \{x \in K_\pfr : v_\pfr(x)=1\}$. Then, using \eqref{eq:height_gcd}, the maximum $M$ in \eqref{eq:v-adic_density} is
    \begin{equation*}
        M=\begin{cases}
            q^{-\beta}, & \alpha \ge \beta > 0,\\
            q^{-\alpha}, & \beta > \alpha > 0,\\
            q^{-\alpha-2\beta}, & 0 > \alpha \ge \beta,\\
            q^{-2\alpha-\beta}, & 0 > \beta > \alpha,\\
            q^{-2\beta}, & \alpha \ge 0 > \beta,\\
            q^{-2\alpha}, & \beta \ge 0 > \alpha,\\
            1, & \alpha > \beta = 0\text{ or }\beta > \alpha = 0,\\
            1, & \alpha = \beta = 0,\ y\text{ or }z \not\equiv 1 \pmod{p},\\
            q^{-\min\{\alpha',\beta'\}}, & \alpha = \beta = 0,\ x_2 \in 1+\pfr^{\alpha'}\Os_\pfr^\times,\ x_3 \in 1+\pfr^{\beta'}\Os_\pfr^\times,\ \alpha',\beta'>0.
        \end{cases}
    \end{equation*}
    It is not hard to see (using \cite[Lemma~8.4]{DP20} in the third and fourth case) that the contribution of the first seven cases to the integral $\omega_{H,\pfr}(X(K_\pfr))$ adds up to
    \begin{equation*}
        \frac 1 q+\frac 1{q^2}+\frac 1 q+\frac 1{q^2}+\frac 1 q+\frac 1 q+2\left(\frac 1 q-\frac 1{q^2}\right) = \frac 6 q.
    \end{equation*}
    The eighth case contributes $1-2/q$, while in the final case, $\alpha'\ge \beta' >0$ contributes $1/q$, and $\beta' > \alpha' > 0$ contributes $1/q^2$. 
\end{proof}

\begin{lemma}\label{lem:tamagawa_number}
    We have
    \begin{equation*}
        \tau_H(X) = \frac{\rho_K^5}{|\Delta_K|}\left(\prod_{v \mid \infty} \omega_{H,v}(X(K_v))\right)\prod_{\pfr \in \Omega_f} \left(\left(1-\frac{1}{\N\pfr}\right)^5\omega_{H,\pfr}(X(K_\pfr))\right).
    \end{equation*}
\end{lemma}

\begin{proof}
    The proof is analogous to \cite[Lemma~8.1]{DP20}, using $\omega_{H,\pfr}(X(K_\pfr))$ from Lemma~\ref{lem:expected_p-adic_densities}.
\end{proof}

Lemma~\ref{lem:expected_archimedean_densities}, Lemma~\ref{lem:expected_p-adic_densities}, and Lemma~\ref{lem:tamagawa_number} show that our Theorem agrees with the Manin--Peyre conjecture.

\section{Parameterization and symmetry}\label{sec:parameterization}

For the proof of our Theorem, our first step is to parameterize the rational points on $X$ using a universal torsor, via its Cox ring. This is a straightforward combination of the parameterization in \cite{Pey98}, \cite{Bre02} over $\QQ$ with the techniques for arbitrary number fields from \cite[\S 2--5]{FP16}. We keep this relatively short, but there is a certain amount of notation to fix.

In Section~\ref{sec:symmetries}, we introduce a symmetry condition. This is less straightforward compared to \cite{Bre02}, due to the consideration of less symmetric height functions, but also due to the presence of different ideal classes. 

\subsection{A universal torsor}

Recall the notation from Section~\ref{sec:notation}, and in particular our convention regarding the indices $i,j,k,l$. 
The Picard group $\Pic(X)\cong \ZZ^5$ has the standard basis $\ell_0,\dots,\ell_4$, with $\ell_0$ the class of $\pi^*\Os_{\PP^2}(1)$ and $\ell_i$ the class of the exceptional divisor of the blow-up of $p_i$. By \cite{BP04,DP19}, $X$ has a Cox ring $\Rs(X)$ over $K$ with ten generators
\begin{equation*}
    a_1,a_2,a_3,a_4,a_{12},a_{13},a_{14},a_{23},a_{24},a_{34},
\end{equation*}
satisfying the five quadratic relations \eqref{eq:torsor} below. The generators correspond to the ten lines on $X$; in particular, $a_i$ corresponds to the exceptional divisor of the blow-up of $p_i$ and hence has degree $\ell_i$, and $a_{ij}$ corresponds to the strict transform of the line through $p_i,p_j$ and has degree $\ell_0-\ell_i-\ell_j$.

As in \cite{FP16}, a universal torsor $Y$ of $X$ is the open subset of $\Spec\Rs(X) \subset \Ab_K^{10}$ where all $(a_i,a_j),(a_i,a_{jk}),(a_{ij},a_{ik})$ are $\ne (0,0)$. As in \cite{Skor93}, \cite[Examples~3.3.4, 4.2.4]{Pey98}, \cite[Proposition~4.1]{BP04}, this is an open subset of the affine cone over the Grassmannian $G(2,5)$, where \eqref{eq:torsor} are the Pl\"ucker equations.

For $P \in \Ps$, we observe that
\begin{equation}\label{eq:def_Ptilde}
    \Pt(a_{1},\dots,a_{34}):=\frac{P(a_{2}a_{3}a_{23},a_{1}a_{3}a_{13},a_{1}a_{2}a_{12})}{a_{1}a_{2}a_{3}a_{4}}
\end{equation}
can be written as a polynomial of anticanonical degree $3l_0-l_1-l_2-l_3-l_4$ using the Pl\"ucker equations \eqref{eq:torsor}. Indeed, any $P \in \Ps$ is a linear combination of the six $P_\sigma$ for $\sigma \in S_3$, and it is easy to check that $\{P_\sigma, Q_\sigma : \sigma \in S_3\}$ gives 
the twelve monomials $\Pt(a_1,\dots,a_{34}) = a_{ij}a_ja_{jk}a_ka_{kl}$; the correspondence is via
\begin{equation}\label{eq:lines}
    \pi^*(Y_i)=a_ja_ka_{jk},\quad \pi^*(Y_j-Y_k)=a_ia_la_{il},
\end{equation}
for $j<k$ and $l=4$. Conversely, given such a $\Pt$, we recover
\begin{equation}\label{eq:P_from_Pt}
    P = \Pt(1,1,1,1,Y_3,Y_2,Y_2-Y_3,Y_1,Y_1-Y_3,Y_2-Y_3).
\end{equation}
(This is essentially the isomorphism between the six-dimensional space of cubics vanishing in $p_1,p_2,p_3,p_4$ and $H^0(X,\omega_X^\vee)$, which is the part of the Cox ring $\Rs(X)$ in anticanonical degree.)

For $v \in \Omega_K$ and $(x_{1v},\dots,x_{34v}) \in K_v^{10}$ and using \eqref{eq:def_Ptilde}, the $v$-adic factors of the height function lifted to the universal torsor are
\begin{equation*}
    \Nt_v(x_{1v},\dots,x_{34v}) = \max_{P \in \Ps}|\Pt(x_{1v},\dots,x_{34v})|_v.
\end{equation*}

The grading of $\Rs(X)$ by $\Pic(X)\cong \ZZ^5$ induces an action of $\GGm^5(K)$ on $Y(K)$, namely
\begin{equation}\label{eq:action}
  \underline u * (a_1,\dots,a_{34}) = (\underline u^{\ell_1}a_1,\dots,
  \underline u^{\ell_0-\ell_3-\ell_4}a_{34})
\end{equation}
for $\underline u =(u_0,\dots,u_4) \in (K^\times)^5$, using the notation $\underline u^{k_0\ell_0+\dots+k_4\ell_4} := u_0^{k_0}\cdots u_4^{k_4}$.

For $\cfrb = (\cfr_0,\dots,\cfr_4) \in \Cs$, we define $u_\cfrb := \N(\cfr_0^3\cfr_1^{-1}\cdots\cfr_4^{-1}) \in \QQ_{>0}$ (corresponding to the anticanonical degree), the fractional ideals $\Os_i =\cfr_i$, $\Os_{jk} = \cfr_0\cfr_j^{-1}\cfr_k^{-1}$ (corresponding to the degrees of $a_i,a_{jk}$) with the subsets $\Os_{i*} = \Os_i^{\ne 0}$, $\Os_{jk*} = \Os_{jk}^{\ne 0}$and the ideals $\afr_i = a_i\Os_i^{-1}$, $\afr_{jk} = a_{jk}\Os_{jk}^{-1}$ .

\begin{prop}\label{prop:torsor_parameterization}
For $\cfrb \in \Cs$ and $B \in \RR_{>0}$, the set $\Ss_\cfrb(B)$ of all
\begin{equation*}
    (a_1,a_2,a_3,a_4,a_{12},a_{13},a_{14},a_{23},a_{24},a_{34}) \in \Os_{1*}\times\dots\times\Os_{34*}
\end{equation*}
satisfying the height condition
\begin{equation}\label{eq:heightK}
    \prod_{v \mid \infty}\Nt_v(a_1^{(v)},\dots,a_{34}^{(v)}) \le u_\cfrb B,
\end{equation}
the torsor equations
\begin{equation}\label{eq:torsor}
    \begin{aligned}
        &a_4a_{14}-a_3a_{13}+a_2a_{12}=0,\\
        &a_4a_{24}-a_3a_{23}+a_1a_{12}=0,\\
        &a_4a_{34}-a_2a_{23}+a_1a_{13}=0,\\
        &a_3a_{34}-a_2a_{24}+a_1a_{14}=0,\\
        &a_{12}a_{34}-a_{13}a_{24}+a_{23}a_{14}=0,
    \end{aligned}  
\end{equation}
and the coprimality conditions
\begin{equation}\label{eq:gcdK}
    \afr_i+\afr_j = \afr_i+\afr_{jk} = \afr_{ij}+\afr_{ik} = \OK
\end{equation}
is invariant under the action of $(\OK^\times)^5$ described by \eqref{eq:action}.

Let $\Ms_\cfrb(B)$ be the set of orbits of this action on $\Ss_\cfrb(B)$. Then
  \begin{equation*}
    N_{U,H}(B) = \sum_{\cfrb \in \Cs} |\Ms_\cfrb(B)|.
  \end{equation*}
\end{prop}

\begin{proof}
    This is analogous to \cite[Lemma~3]{Bre02} (without the symmetry from \cite[Lemma~2]{Bre02}) combined with \cite[Lemma~4.3]{FP16}. Note that the coprimality condition \eqref{eq:height_gcd} in our definition of admissible height function ensures that the product over the height factors at the finite places contributes
    \[\N(\cfr_0^{-3}\cfr_1\cfr_2\cfr_3\cfr_4)=u_\cfrb^{-1},\]
    accounting for the equivalence between \eqref{eq:heightK} and the previous condition $H(x) \le B$.
\end{proof}

\subsection{Symmetry}\label{sec:symmetries}

We introduce symmetry conditions similar to \cite[Lemma~4]{Bre02} (but not the $S_4$-symmetry from \cite[Lemma~2]{Bre02}), leading to a decomposition of the main term as a sum of five terms. Since we work with more general (and in general less symmetric) height functions compared to \cite{Bre02}, we cannot argue directly that all five terms give the same contribution. Nonetheless, this conclusion remains true in our general case, and will be deduced at the very end of the proof of our Theorem, using a comparison of the archimedean densities in Lemma~\ref{lem:archimedean_density_under_symmetry}.

To set up the symmetry condition, we consider for $i \in \{1,2,3,4\}$ the involution $s_i$ on our ten coordinates $(x_1,\dots,x_{34})$ that changes the sign of $x_1$ (for $i=1$), $x_{12}$ (for $i=2$), $x_{34}$ (for $i=3$), or $x_4$ (for $i=4$), fixes the remaining three of the four coordinates $x_i,x_{ij}$ and exchanges $x_{jk}$ with $x_l$. Let $S=\{\id,s_1,s_2,s_3,s_4\}$. 

\begin{remark*}
    Here $s_i$ corresponds to that element of the Weyl group $S_5$ of the root system $A_4$ associated with the quintic del Pezzo surface (see \cite[Chapter~IV]{Manin}) that exchanges the four pairwise skew lines corresponding to $a_i,a_j,a_k,a_l$ with the four pairwise skew lines corresponding to $a_i,a_{kl},a_{jl},a_{jk}$; this is the reflection on the hyperplane orthogonal to the root $\ell_0-\ell_j-\ell_k-\ell_l$ in $\Pic(X)$, which fixes the anticanonical class. Note that $s_4$ also corresponds to the standard quadratic Cremona transformation of $\PP^2_K$ based in $p_1,p_2,p_3$ (fixing $p_4$), and similarly for $s_1,s_2,s_3$.
\end{remark*}

For $s \in S$ and $P \in \Ps$, let $P^{(s)} \in \OK[Y_1,Y_2,Y_3]$ be the polynomial such that the corresponding $\widetilde{P^{(s)}}$ as in \eqref{eq:def_Ptilde} satisfies $\widetilde{P^{(s)}}(x_1,\dots,x_{34}) = \Pt(s(x_1,\dots,x_{34}))$. (Since this expression has anticanonical degree, such a $P^{(s)}$ exists by \eqref{eq:P_from_Pt}.)

Let
\begin{equation}\label{eq:def_P^s}
    \Ps^{(s)} = \{P^{(s)} : P \in \Ps\},
\end{equation}
which has the same properties as $\Ps$, as explained in Section~\ref{sec:heights}. 

\begin{remark*}
    As an illustration, the monomial $x_{ij}x_jx_{jk}x_kx_{kl}$ corresponding to one of $P_{\sigma}, Q_{\sigma}$ is mapped by $s_l$ to $\pm x_{ji}x_ix_{ik}x_kx_{kl}$, while $s_k$ maps it to $\pm x_{ik}x_kx_{kl}x_lx_{lj}$. This shows that for $\Ps=\{P_{\sigma},Q_{\sigma}: \sigma \in S_3\}$ the sets $\Ps^{(s)}$ and $\Ps$ agree up to signs for all $s \in S$, but for $\Ps=\{P_{\sigma}: \sigma \in S_3\}$, the analogous statement fails as some of the $P_{\sigma}$ are mapped to $\pm Q_{\sigma'}$.
\end{remark*}

For $s=s_i \in S$, let $s(\cfrb) := (\cfr_0',\dots,\cfr_4')$ with $\cfr_0':=\cfr_0^2\cfr_j^{-1}\cfr_k^{-1}\cfr_l^{-1}$, $\cfr_i':=\cfr_i$, and $\cfr_j':=\cfr_0\cfr_k^{-1}\cfr_l^{-1}$; for $s=\id$, let $s(\cfrb)=\cfrb$. Then
\begin{equation}\label{eq:def_C^s}
    \Cs^{(s)} := \{s(\cfr) : \cfr \in \Cs\}
\end{equation}
is again a system of representatives of $\Cl_K^5$ because of the above correspondence of $s$ to an involution of $\Pic(X)$.

Similarly, $s \in S$ defines an involution on $(\OK^\times)^5$ as follows. For $s=s_i \in S$, let $s(\underline u) = (u_0',\dots,u_4')$ with $u_0':=u_0^2u_j^{-1}u_k^{-1}u_l^{-1}$, $u_i':=u_i$, and $u_j':=u_0u_k^{-1}u_l^{-1}$; for $s=\id$, let $s(\underline u) = \underline u$. Since this is compatible with the action \eqref{eq:action} of $(\OK^\times)^5$ in the sense that
\begin{equation}\label{eq:actions_compatible}
    s(\underline u * (a_1,\dots,a_{34})) = s(\underline u) * s(a_1,\dots,a_{34}),
\end{equation}
we observe that $s \in S$ maps the orbit of $(a_1,\dots,a_{34})$ to the orbit of $s(a_1,\dots,a_{34})$.

\begin{lemma}\label{lem:weyl_group_symmetry}
    Let $\Msover_{\cfrb}^{(s)}(B)$ be the set of all $(\OK^\times)^5$-orbits of $(a_1,\dots,a_{34}) \in \Os_{1*}\times\dots\times\Os_{34*}$ satisfying \eqref{eq:torsor}, \eqref{eq:gcdK}, the four symmetry conditions
    \begin{equation}\label{eq:symmetryK}
        |N(a_ia_ja_k)| \le |N(a_{ij}a_{ik}a_{jk})|,
    \end{equation}
    and the height condition \eqref{eq:heightK} with $\Ps$ replaced by $\Ps^{(s)}$ in the definition of $\Nt_v$.

    Let $\Msunder_\cfrb^{(s)}(B)$ be defined analogously, with \eqref{eq:symmetryK} replaced by
    \begin{equation}\label{eq:symmetry_strict}
        |N(a_ia_ja_k)| < |N(a_{ij}a_{ik}a_{jk})|.
    \end{equation}

    Then
    \begin{equation*}
        \sum_{s \in S} |\Msunder_{s(\cfrb)}^{(s)}(B)| \le |\Ms_\cfrb(B)| \le \sum_{s \in S} |\Msover_{s(\cfrb)}^{(s)}(B)|.
    \end{equation*}
\end{lemma}

\begin{proof}
  We consider the five subsets of $\Ms_\cfrb(B)$ where a certain one of the five integers $|N(a_1a_2a_3a_4)|$ and $|N(a_ia_{jk}a_{jl}a_{kl})|$ is their minimum. If $|N(a_1a_2a_3a_4)|$ is the minimum, this subset clearly is $\Msover_\cfrb^{(\id)}(B)$.
  
  If it is $|N(a_ia_{jk}a_{jl}a_{kl})|$, an application of $s=s_i$ (which respects the $(\OK^\times)^5$-orbits by \eqref{eq:actions_compatible}) leads to condition \eqref{eq:symmetryK} and replaces $\Os_j$ by $\Os_{kl}$, which is $\Os_j$ computed with respect to $s(\cfrb)$, and analogously for $\Os_{kl}$. Furthermore, it leaves \eqref{eq:torsor}, \eqref{eq:gcdK} (with respect to $s(\cfrb)$) invariant, and replaces $\Ps$ by $\Ps^{(s)}$ in the definition of the height condition \eqref{eq:heightK} (with $u_\cfrb=u_{s(\cfrb)}$). 

  We note the left-hand side of our claim does not count the orbits with equality in \eqref{eq:symmetryK}, while the right-hand side counts these orbits multiple times.
\end{proof}

\subsection{Dependent coordinates}

Three of the five torsor equations allow us to express three of the $a_{ij}$ in terms of the remaining variables; the remaining two equations are redundant if $a_1,\dots,a_4$, for example, are nonzero. To obtain integral $a_{ij}$, we will see that two congruence conditions are enough (modulo fractional ideals; recall our convention from Section~\ref{sec:notation}).

For simplicity, we write
\begin{equation*}
    \ab'=(a_1,\dots, a_4), \quad \ab''=(a_{12},a_{13},a_{14},a_{23},a_{24},a_{34}),
\end{equation*}
and
\begin{equation*}
    \Os_*'=\Os_{1*}\times\dots\times\Os_{4*}, \quad 
    \Os'' = \Os_{12}\times\Os_{13}\times\Os_{14}\times\Os_{23}\times\Os_{24}\times\Os_{34}.
\end{equation*}

\begin{lemma}\label{lem:dependent_aij}
    Let $\ab' \in \Os_*'$, and $(a_{12},a_{23},a_{34}) \in \Os_{12}\times\Os_{23}\times\Os_{34}$.
    If 
    \begin{equation}\label{eq:mod_a4}
        \congr{a_3a_{23}}{a_1a_{12}}{a_4\Os_{24}}
    \end{equation}
    and
    \begin{equation}\label{eq:mod_a1}
        \congr{a_4a_{34}}{a_2a_{23}}{a_1\Os_{13}}
    \end{equation}
    hold, then we obtain unique
    \begin{equation}\label{eq:dependent_aij}
        \begin{aligned}
            a_{13}&=\frac{a_2a_{23}-a_4a_{34}}{a_1},\\
            a_{24}&=\frac{a_3a_{23}-a_1a_{12}}{a_4},\\
            a_{14}&=\frac{a_2a_3a_{23}-a_3a_4a_{34}-a_1a_2a_{12}}{a_1a_4}
        \end{aligned}
    \end{equation}
    satisfying the torsor equations \eqref{eq:torsor}, with $a_{13} \in \Os_{13}$ and $a_{24} \in \Os_{24}$. If additionally $\afr_1+\afr_4=\OK$, then $a_{14} \in \Os_{14}$. If (3.12) or (3.13) does not hold, \eqref{eq:torsor} has no solution satisfying $\ab'' \in \Os''$.
\end{lemma}

\begin{proof}
    Indeed, $a_{24}$ and \eqref{eq:mod_a4} are obtained from the second torsor equation, $a_{13}$ and \eqref{eq:mod_a1} from the third equation, and $a_{14}$ from the first and third equation, where $a_{14} \in \Os_{14}$ because of both \eqref{eq:mod_a4} and \eqref{eq:mod_a1} if $\afr_1+\afr_4=\OK$. An easy computation (using that all $a_i$ are nonzero) shows that all torsor equations as in \eqref{eq:torsor} hold.
\end{proof}

\subsection{Construction of a fundamental domain}\label{sec:fundamental_domain}

Instead of considering orbits for the torus action on the universal torsor $Y$ over $X$, we construct a good fundamental domain for (a slightly modified version of) this action, as in \cite[\S 5]{FP16}.

Consider the action of the subgroup $U_K \times (\OK^\times)^4$ of $(\OK^\times)^5$ (see Section~\ref{sec:notation}) on $Y(K)\cap(K^\times)^{10}$ described by \eqref{eq:action}. 
Let $\Fs$ be a fundamental domain for this action. Let $\Mover_\cfrb^{(s)}(B)$ be the number of $(a_1,\dots,a_{34}) \in \Os_1\times\dots\times\Os_{34} \cap \Fs$ satisfying \eqref{eq:heightK} with $\Ps$ replaced by $\Ps^{(s)}$, \eqref{eq:torsor}, \eqref{eq:gcdK}, and \eqref{eq:symmetryK}. Since $U_K \times (\OK^\times)^4$ has index $|\mu_K|$ in $(\OK^\times)^5$, we have
\begin{equation}\label{eq:orbit_to_fundamental_domain}
     |\Msover_\cfrb^{(s)}(B)| = \frac{1}{|\mu_K|}\cdot |\Mover_\cfrb^{(s)}(B)|.
\end{equation}
Similarly (with \eqref{eq:symmetryK} replaced by \eqref{eq:symmetry_strict}), we define $\Munder_\cfrb^{(s)}(B)$.

In Sections~\ref{sec:restrictions} and \ref{sec:main_contribution}, we estimate $|\Mover_\cfrb^{(\id)}(B)|$ for arbitrary $\Ps$ and $\cfrb$. This will give us estimates for all $|\Mover_{s(\cfrb)}^{(s)}(B)|$ upon replacing $\Ps$ by $\Ps^{(s)}$ and $\cfrb$ by $s(\cfrb)$. Furthermore, it will be clear that the same estimations hold for $|\Munder_{s(\cfrb)}^{(s)}(B)|$. This will be used when putting all the parts together in the final steps of the proof of our Theorem.

Now we construct such a fundamental domain $\Fs$, using $F(\infty), F(B)$ as in \cite[\S 5]{FP16}.
For $\ab'=(a_1,\dots, a_4) \in (K^\times)^4$ and $(x_{12v},x_{23v},x_{34v}) \in (K_v^\times)^3$, let
\begin{equation*}
  \Nt_v(\ab';x_{12v},x_{23v},x_{34v}):=\Nt_v(a_1^{(v)},\dots,a_4^{(v)},x_{12v},x_{13v},x_{14v},x_{23v},x_{24v},x_{34v}),
\end{equation*}
where
\begin{equation}\label{eq:dependent_xijv}
    \begin{aligned}
        x_{13v} &= \frac{a_2^{(v)}x_{23v}-a_4^{(v)}x_{34v}}{a_1^{(v)}},\\
        x_{24v} &= \frac{a_3^{(v)}x_{23v}-a_1^{(v)}x_{12v}}{a_4^{(v)}}, \\
        x_{14v} &= \frac{a_2^{(v)}a_3^{(v)}x_{23v}-a_3^{(v)}a_4^{(v)}x_{34v}-a_1^{(v)}a_2^{(v)}x_{12v}}{a_1^{(v)}a_4^{(v)}}
    \end{aligned}
\end{equation}
are expressed in terms of $\ab',x_{12v},x_{23v},x_{34v}$ via the torsor equations (see \eqref{eq:dependent_aij}). Let $S_F(\ab';\infty)$ be the set
\begin{equation*}
  \bigwhere{(x_{12v},x_{23v},x_{34v})_v \in \prod_{v\mid\infty} (K_v^\times)^3}
  {&\frac{1}{3}(\log \Nt_v(\ab';x_{12v},x_{23v},x_{34v}))_v \in F(\infty),\\
  &\text{$x_{13v},x_{14v},x_{24v}$ as in \eqref{eq:dependent_xijv} are $\ne 0$}}.
\end{equation*}
(Here, $\Nt_v(\ab';x_{12v},x_{23v},x_{34v}) \ne 0$ since all $\widetilde{P_\sigma}$ are nonzero and linear combinations of $\Pt$ for $P \in \Ps$, so that not all these $\Pt$ can vanish.) Then
\begin{equation*}
  \Fs_0(\ab'):=\{(a_{12},a_{23},a_{34}) \in K^3 : \sigma(a_{12},a_{23},a_{34}) \in S_F(\ab';\infty)\}
\end{equation*}
is a fundamental domain for the action of $U_K$ by scalar multiplication on the set of $(a_{12},a_{23},a_{34}) \in (K^\times)^3$ such that $a_{13},a_{14},a_{24}$ as in \eqref{eq:dependent_aij} are also nonzero.

Let $\Fs_1$ be a fundamental domain for the action of $\OK^\times$ on
$K^\times$ as in \cite[\S 5]{FP16}, satisfying
\begin{equation}\label{eq:ai_conjugates}
  |a|_v \asymp N(a)^{d_v/d}
\end{equation}
for all $a \in \Fs_1$ and all $v \mid \infty$.

Ignoring the dependent coordinates $a_{13},a_{14},a_{24}$, our action \eqref{eq:action} induces an action of $U_K \times (\OK^\times)^4$ on
\begin{equation*}
    (K^7)_* :=\{(a_1,a_2,a_3,a_4,a_{12},a_{23},a_{34}) \in (K^\times)^7 : \text{$a_{13},a_{14},a_{24}$ as in \eqref{eq:dependent_aij} are $\ne 0$}\},
\end{equation*}
which is free and has fundamental domain
\begin{equation*}
  \Fs':=\{(a_1,a_2,a_3,a_4,a_{12},a_{23},a_{34}) \in (K^7)_* : \ab' \in \Fs_1^4,\ 
  (a_{12},a_{23},a_{34}) \in \Fs_0(\ab')\}.
\end{equation*}
Hence our fundamental domain $\Fs$ is given by the set of all $(a_1,\dots,a_{34}) \in (K^\times)^{10}$ with $(a_1,a_2,a_3,a_4,a_{12},a_{23},a_{34}) \in \Fs'$ and $a_{13},a_{14},a_{24}$ as in \eqref{eq:dependent_aij}.

An element $(x_{12v},x_{23v},x_{34v})_v \in S_F(\ab';\infty)$ satisfies the
height condition
\begin{equation*}
  \prod_{v \mid \infty} \Nt_v(\ab';x_{12v},x_{23v},x_{34v}) \le u_\cfrb B
\end{equation*}
if and only if it lies in $S_F(\ab';u_\cfrb B)$ defined as
\begin{equation*}
  \bigwhere{(x_{12v},x_{23v},x_{34v})_v \in \prod_{v\mid\infty} (K_v^\times)^3\!}
  {&\frac{1}{3}(\log \Nt_v(\ab';x_{12v},x_{23v},x_{34v}))_v \in F((u_\cfrb B)^{\frac{1}{3d}})\\
  &\text{$x_{13v},x_{14v},x_{24v}$ as in \eqref{eq:dependent_xijv} are $\ne 0$}}.
\end{equation*}
Let
\begin{equation*}
  \Fs_0(\ab';u_\cfrb B):=\{(a_{12},a_{23},a_{34}) \in K^3 : \sigma(a_{12},a_{23},a_{34}) \in S_F(\ab';u_\cfrb B)\}.
\end{equation*}
By construction of $F(\infty)$, for all $(x_{12v},x_{23v},x_{34v})_v \in S_F(\ab';u_\cfrb B)$ and all $v,w\mid\infty$, we have
\begin{equation*}
  \Nt_v(\ab';x_{12v},x_{23v},x_{34v})^{\frac{1}{d_v}} \asymp \Nt_w(\ab';x_{12w},x_{23w},x_{34w})^{\frac{1}{d_w}} \ll B^{\frac{1}{d}};
\end{equation*}
since $\Ps$ generates the space of anticanonical polynomials, this implies
\begin{equation}\label{eq:height_factors_monomials}
    |x_{ijv}a_j^{(v)}x_{jkv}a_k^{(v)}x_{klv}|_v \ll B^{\frac{d_v}{d}}.
\end{equation}

\section{Restrictions of the counting problem}\label{sec:restrictions}

\subsection{Restricting the set of $\ab''$}

A key step in our proof is Proposition~\ref{prop:E_WK} below, where we will restrict the variables $a_{ij}$ essentially to their typical sizes given as follows. This allows us to simplify many of the following steps already over $\QQ$, and more importantly, it allows us to generalize the proof to arbitrary number fields.

For $v \mid \infty$, let 
\begin{equation*}
    B_{ijv} := \frac{(u_\cfrb B|N(a_1\dots a_4)|)^{\frac{1}{3d}}}{\sigma_v(a_ia_j)} \in K_v,
\end{equation*}
and
\begin{equation}\label{eq:Bij}
    B_{ij} := \prod_{v \mid \infty} |B_{ijv}|_v = \frac{(u_\cfrb B|N(a_1\dots a_4)|)^{\frac{1}{3}}}{|N(a_ia_j)|}.
\end{equation}
We consider the condition
\begin{equation}\label{eq:bound_SFv}
    |x_{ijv}|_v \le |WB_{ijv}|_v
\end{equation}
for all $i,j$ (using \eqref{eq:dependent_xijv} for $x_{13v},x_{14v},x_{24v}$) and $v \mid \infty$. 

Given $\mathbf{a}' \in \Os_*'$ and $W \ge 1$, let 
    \begin{equation*}
        S_F^{(W)} = S_F^{(W)}(\ab';u_\cfrb B)=\{(x_{12v},x_{23v},x_{34v})_v \in S_F(\ab';u_\cfrb B) : \text{\eqref{eq:bound_SFv} for all $i,j,v \mid \infty$}\}
    \end{equation*}
    and
    \begin{equation*}
        \Fs_0^{(W)} = \Fs_0^{(W)}(\ab';u_\cfrb B)=\{(a_{12},a_{23},a_{34}) \in K^3 : \sigma(a_{12},a_{23},a_{34}) \in S_F^{(W)}(\ab';u_\cfrb B)\}.
    \end{equation*}
We introduce the condition
\begin{equation}\label{eq:in_F0W}
    (a_{12},a_{23},a_{34}) \in \Fs_0^{(W)}(\ab';u_\cfrb B)
\end{equation}
to define
\begin{equation*}
    A = A(W,\ab',\cfrb,B) := \{\ab'' \in \Os'' : \eqref{eq:torsor}, \eqref{eq:gcdK}, \eqref{eq:in_F0W}\}
\end{equation*}
and its subset
\begin{equation*}
    \Aover = \Aover(W,\ab',\cfrb,B) := \{\ab'' \in \Os'' :  \eqref{eq:torsor}, \eqref{eq:gcdK}, \eqref{eq:symmetryK}, \eqref{eq:in_F0W}\}.
\end{equation*}
with the symmetry condition. We observe that these sets are empty unless all
\begin{equation}\label{eq:bound_ai_B}
    \N\afr_j\ll B.
\end{equation}
Indeed, \eqref{eq:heightK} and the fact that all $\afr_i,\afr_{kl}$ are nonzero ideals implies
\begin{equation*}
    \N\afr_j \le \N(\afr_{ij}\afr_j\afr_{jk}\afr_k\afr_{kl}) 
    = u_\cfrb^{-1} \prod_{v\mid \infty}|a_{ij}a_ja_{jk}a_ka_{kl}|_v \ll B,
\end{equation*}
where the last inequality holds by \eqref{eq:height_factors_monomials}.

We will often require the following lemma.

\begin{lemma}\label{lem:sum_log}
    For any $1 \le t_1<t_2$, we have
    \[\sums{\afr \in \IK\\t_1 \le \N\afr \le t_2} \frac{1}{\N\afr} \ll \log\left(\frac{t_2}{t_1}\right)+1.\]
\end{lemma}

\begin{proof}
    By a dyadic decomposition, it suffices to prove this for $t_2=2t_1$. Here, the statement follows from \[\#\{\afr \in \IK: \N\afr \le t\} \ll t,\]
    which holds for $t \ge 1$ by the ideal theorem (see \cite[Lemma~2.5]{DF14}, for example).
\end{proof}

In particular, in combination with \eqref{eq:bound_ai_B}, we get 
\begin{equation}\label{eq:sum_ai}
    \sums{\ab' \in \Os_*'\cap \Fs_1^4\\\eqref{eq:bound_ai_B}} \frac{1}{|N(a_1a_2a_3a_4)|} \ll (\log B)^4
\end{equation}
since every $\afr_i = a_i\Os_i^{-1}$ runs through all nonzero ideals in the same class as $\Os_i^{-1}$ if $a_i \in \Os_{i*} \cap \Fs_1$, and $\N\afr_i \asymp |N(a_i)|$.

\begin{prop}\label{prop:E_WK}
    For $W \ge 1$, we have
    \begin{equation*}
        |\Mover_\cfrb^{(\id)}(B)| = \sums{\ab' \in \Os_*' \cap \Fs_1^4} |\Aover(W,\ab',\cfrb,B)|
        + O\left(\frac{B(\log B)^4}{W}\right).
    \end{equation*}
\end{prop}

\begin{proof}
    We begin by introducing some notation: For a fixed solution $(\ab',\ab'')$ and a valuation $v\mid\infty$, we write $z_{ijv}=\sigma_v(a_{ij})B_{ijv}^{-1}$. Because of \eqref{eq:height_factors_monomials}, our height condition implies $\vert z_{ijv}z_{jkv}z_{kl v}\vert_v \ll 1$ while the torsor equations become $z_{ijv} \pm z_{ikv} \pm z_{ilv}=0$.

    We claim that for some $W_v \gg 1$ and for some $i$, we have 
    \begin{equation*}
        |z_{ijv}|_v, |z_{ikv}|_v,|z_{ilv}|_v \ll \frac{1}{W_v^2},\quad |z_{jkv}|_v, |z_{jlv}|_v, |z_{klv}|_v\ll W_v.
    \end{equation*}
    This is obvious with $W_v =1$ if all the $|z_{ijv}|$ are $\le 1$. Otherwise, without loss of generality, suppose that $|z_{12v}|_v$ is maximal and equals $W_v > 1$. Then by the first torsor equation, without loss of generality, we must also have $|z_{13v}|_v \asymp W_v$. Then the height conditions imply $|z_{24v}|_v, |z_{34v}|_v \ll W_v^{-2}$, and the fourth torsor equation then also implies $|z_{14v}|_v \ll W_v^{-2}$, proving our claim.

    For each $(\ab',\ab'')$, let $V_i$ be the set of $v\mid\infty$ with this choice of index $i$ (if it is not unique, we choose one), and let $W_i=\prod_{v \in V_i} W_v \ge 1$ and $W_0=\max_{v\mid\infty} W_v^{1/d_v}$.

    The condition $\ab'' \not \in \Fs_0^{(W)}(\ab';u_\cfrb B)$ thus implies $W_0 \gg W$. By a dyadic decomposition, it suffices to prove that the number of $(\ab',\ab'')$ with $W_0 \sim W$ is $\ll W^{-1}B(\log B)^4$ for any fixed $W \ge 1$.

    By a further dyadic decomposition of the range of the $W_v$, it suffices to prove that for a fixed choice of $(W_v)_{v\mid\infty}$, the number of $(\ab',\ab'')$ is $\ll W_0^{-2}B(\log B)^4$, where still $W_0=\max_{v\mid\infty} W_v^{1/d_v}$. Indeed, as we may always choose the $W_v$ to be powers of two, the number of such choices for $(W_v)_{v\mid\infty}$ is bounded by $(\log W_0)^{|\Omega_\infty|} \ll W_0$.

    Moreover, by symmetry it suffices to consider the contribution from those solutions where $W_2=\max_i W_i$. Note that $W_2 \ge W_0$.

    In view of Lemma~\ref{lem:dependent_aij}, $a_{13},a_{14},a_{24}$ are determined by the other variables using the torsor equation \eqref{eq:torsor}, and they can exist only if the congruences \eqref{eq:mod_a4} and \eqref{eq:mod_a1} are satisfied. Note that the congruence for $a_3a_{23}$ modulo $a_4\Os_{24}$ is equivalent to a congruence for $a_{23}$ modulo $a_4\Os_{24}\Os_3^{-1}=\afr_4\Os_{23}$, and similarly for $a_{34}$. Therefore, writing $a_{23} \Mod{\afr_4\Os_{23}}$ and $a_{34} \Mod{\afr_1\Os_{34}}$ to denote these congruence conditions (where we suppress the dependence of the residue class on the remaining variables, as all estimates will be uniform in them), we can bound the number of solutions by
    \[\sum_{a_{12}} \sums{a_{23}\\\Mod{\afr_4\Os_{23}}} \sums{a_{34}\\\Mod{\afr_1\Os_{34}}} 1,\]
    where the sum is restricted to $|a_{ij}|_v \ll W_v^{-2}|B_{ijv}|_v$ for $v \in V_i,V_j$ and to $|a_{ij}|_v \ll W_v|B_{ijv}|_v$ for $v \in V_k,V_{l}$. 
    
    By \cite[Lemma~7.1]{FP16}, for fixed $a_{12}$ and $a_{23}$, the number of $a_{34}$ is uniformly bounded by
    \begin{equation} \label{eq:upperbound_a34}
    \ll \frac{W_1W_2}{W_3^2W_4^2} \cdot \frac{B_{34}}{|N(a_1)|} +1.    
    \end{equation}
    But using the identity
    \[\frac{|N(a_i)|}{B_{jk}}=\left(\frac{|N(a_ia_ja_k)|}{B_{ij}B_{ik}B_{jk}}\right)^{\frac{1}{3}}\]
    by \eqref{eq:Bij}, the symmetry condition \eqref{eq:symmetryK} turns into $|N(a_i)| \ll W_i^{-1}W_j^{-1}W_k^{-1}W_lB_{jk}$. In particular, we have
    \[|N(a_1)| \ll \frac{W_2}{W_1W_3W_4} \cdot B_{34}\]
    and hence the number of $a_{34}$ is bounded by
    \[\ll \frac{W_1W_2}{W_3W_4} \cdot \frac{B_{34}}{|N(a_1)|}\]
    where we have chosen the right-hand side as an upper bound for both summands in \eqref{eq:upperbound_a34}.

    Similarly, for fixed $a_{12}$, the number of $a_{23}$ is bounded by
    \[\ll \frac{W_1W_4}{W_2^2W_3^2} \cdot \frac{B_{23}}{|N(a_4)|} +1 \ll \frac{W_1}{W_2W_3} \cdot \frac{B_{23}}{|N(a_4)|}\]
    upon using the bound $|N(a_4)| \ll W_1W_2^{-1}W_3^{-1}W_4^{-1} B_{23}$ and the assumption $W_4 \le W_2$.
    
    Finally, the number of $a_{12}$ is $\ll W_1^{-2}W_2^{-2}W_3W_4 B_{12} $ by \cite[Lemma~7.2]{FP16} (where no ``$+1$'' appears since we are counting nonzero elements without a congruence condition).
    
    The total contribution therefore is
    \begin{align*}
        &\ll \frac{W_3W_4}{W_1^2W_2^2} \cdot \frac{W_1}{W_2W_3} \cdot \frac{W_1W_2}{W_3W_4} \cdot \frac{B_{12}B_{23}B_{34}}{|N(a_1)N(a_4)|}\\ 
        &=\frac{1}{W_2^2W_3} \cdot \frac{u_\cfrb B}{|N(a_1a_2a_3a_4)|}  \ll \frac{1}{W_0^2} \cdot \frac{B}{|N(a_1a_2a_3a_4)|},
    \end{align*}
    leading to the desired bound after summing over $\ab'$ using \eqref{eq:sum_ai}.
\end{proof}

\begin{lemma}
    For $W \ge 1$, let $\ab' \in \Os_{*}'$ be such that $\Aover(W,\ab',\cfrb,B)$ is nonempty.
    Then $\ab'$ satisfies
    \begin{equation}\label{eq:bound_W}
        |N(a_i^2a_j^2a_k^2a_l^{-1})| \le W^{3d} u_\cfrb  B
    \end{equation}
    and
    \begin{equation}\label{eq:bound_ai}
        |N(a_i)| \le W^d B_{jk}.
    \end{equation}
\end{lemma}

\begin{proof}
    Choose some $\ab'' \in \Aover(W,\ab',\cfrb,B)$.
    Taking the product of the conditions \eqref{eq:bound_SFv} defining $S_F^{(W)}$ gives $|N(a_{ij})| \le W^dB_{ij}$.
    Combining this with the symmetry condition \eqref{eq:symmetryK} and \eqref{eq:Bij} shows
    \begin{equation*}
        |N(a_ia_ja_k)| \le |N(a_{ij}a_{ik}a_{jk})| \le W^{3d}B_{ij}B_{ik}B_{jk} = \frac{W^{3d}u_\cfrb B|N(a_1\cdots a_4)|}{|N(a_i^2a_j^2a_k^2)|},
    \end{equation*}
    which gives the first result. The second one is equivalent to it, in view of \eqref{eq:Bij}.
\end{proof}

Now we can deduce the following uniform upper bound for the number of $\ab''$.

\begin{lemma}\label{lem:bound_a''}
    For $W \ge 1$ and $\ab' \in \Os_*'$, we have
    \begin{equation*}
        |\Aover(W,\ab',\cfrb,B)| \ll \frac{W^{3d}B}{|N(a_1\dots a_4)|}.
    \end{equation*}
\end{lemma}

\begin{proof}
  As in the proof of Proposition~\ref{prop:E_WK}, an application of Lemma~\ref{lem:dependent_aij} shows that
  \begin{equation*}
      |\Aover(W,\ab',\cfrb,B)| \ll \sums{a_{12}\\|a_{12}|_v \le |WB_{12v}|_v}
      \sums{a_{23}\Mod{\afr_4\Os_{23}}\\|a_{23}|_v \le |WB_{23v}|_v}
      \sums{a_{34}\Mod{\afr_1\Os_{34}}\\|a_{34}|_v \le |WB_{34v}|_v} 1.
  \end{equation*}
  By \cite[Lemma~7.1]{FP16}, the sum over $a_{34}$ is $\ll W^d B_{34}/|N(a_1)|+1$; by \eqref{eq:bound_ai}, we can leave out the $+1$. For $a_{23}$, we argue similarly. By \cite[Lemma~7.2]{FP16}, the sum over $a_{12}$ is $\ll W^d B_{12}$. Hence the total bound is
  \begin{equation*}
      \ll \frac{W^{3d}B_{12}B_{23}B_{34}}{|N(a_1a_4)|} \ll \frac{W^{3d}B}{|N(a_1\dots a_4)|}.\qedhere
  \end{equation*}
\end{proof}

\subsection{Restricting the set of $\ab'$}\label{sec:restrict_a'}

From here, let $T_2 = \exp(c_2 \log B / \log \log B)$ and $W=(\log\log B)^{1/(3d+1)}$, as described in Section~\ref{sec:notation}. We show (Lemma~\ref{lem:restrict_a'}) that we may restrict the summation over $\ab'$ by bootstrapping \eqref{eq:bound_W} to the conditions
\begin{equation}\label{eq:bound_T2}    
    \N(\afr_i^2\afr_j^2\afr_k^2\afr_l^{-1}) \le \frac{B}{T_2^d}.
\end{equation}
While this result is analogous to \cite[Proposition~1]{Bre02}, its proof relies on our Proposition~\ref{prop:E_WK}. This restriction will be crucial to estimate the summation of error terms over $\ab'$ in Proposition~\ref{prop:remove_symmetry} and Proposition~\ref{prop:errors_volumes}.

\begin{lemma}\label{lem:bound_a'_B}
    For $\ab' \in \Os_*'$ satisfying \eqref{eq:bound_T2}, the inequality \eqref{eq:bound_ai_B} holds.
\end{lemma}

\begin{proof}
    By symmetry, it is enough to consider $j=1$. Multiplying all three bounds from \eqref{eq:bound_T2} with $l\ne 1$ gives $\N(\afr_1^6\afr_2^3\afr_3^3\afr_4^3) \le B^3T_2^{-3d} \le B^3$. Since $\N(\afr_2\afr_3\afr_4) \ge 1$, this implies $\N(\afr_1) \le B^{1/2}$.
\end{proof}

\begin{lemma}\label{lem:restrict_a'}
    We have
    \begin{equation*}
        |\Mover_\cfrb^{(\id)}(B)| = \sums{\ab' \in \Os_*' \cap \Fs_1^4\\\eqref{eq:bound_T2}} |\Aover(W,\ab',\cfrb,B)| + O\left(\frac{B(\log B)^4}{(\log \log B)^{\frac{1}{3d+1}}}\right).
    \end{equation*}
\end{lemma}

\begin{proof}
    In view of \eqref{eq:bound_W}, it is enough to sum the number of $\ab''$ as in Lemma~\ref{lem:bound_a''} over $\ab'$ satisfying
    \begin{equation}\label{eq:bound_al}
        \frac{|N(a_i^2a_j^2a_k^2)|}{W^{3d}u_\cfrb B} \le |N(a_l)| \ll \frac{T_2^d|N(a_i^2a_j^2a_k^2)|}{u_\cfrb B}.
    \end{equation}
    By Lemma~\ref{lem:bound_a''}, their contribution is bounded by
    \begin{equation*}
        \sums{\ab' \in \Os_* \cap \Fs_1^4\\\eqref{eq:bound_al}}\frac{W^{3d}B}{|N(a_1\cdots a_4)|} \ll W^{3d} B(\log B)^3(\log W^{3d}T_2^d) \ll W^{3d}\frac{B(\log B)^4}{\log \log B}.
    \end{equation*}
    Indeed, after replacing the summation over $\ab'$ by a summation over ideals $\afrb'$ (as for \eqref{eq:sum_ai}), we apply Lemma~\ref{lem:sum_log} (using \eqref{eq:bound_al} for the summation over $\afr_l$, and \eqref{eq:bound_ai_B} for $\afr_i,\afr_j,\afr_k$).

    Finally, both this error term and the previous one from Proposition~\ref{prop:E_WK} are satisfactory in view of our choice of $W$.
\end{proof}

\subsection{Removing the symmetry conditions}

Now we may remove \eqref{eq:symmetryK} as in Lemma~\ref{lem:weyl_group_symmetry} from the set $\Aover$ as in Proposition~\ref{prop:E_WK}.

\begin{prop}\label{prop:remove_symmetry}
    We have
    \begin{equation*}
        |\Mover_\cfrb^{(\id)}(B)| = \sums{\ab' \in \Os_*' \cap \Fs_1^4\\\eqref{eq:bound_T2}} |A(W,\ab',\cfrb,B)| + O\left(\frac{B(\log B)^4}{(\log \log B)^{\frac{1}{3d+1}}}\right).
    \end{equation*}
\end{prop}

\begin{proof}
    We must show that the number of points violating \eqref{eq:symmetryK} is negligible. 
    
    Such a point satisfies $|N(a_{ij}a_{ik}a_{jk})| < |N(a_ia_ja_k)|$ for some $i,j,k$. Dividing by $B_{ij}B_{ik}B_{jk} = u_\cfrb B/|N(a_ia_ja_ka_l^{-1})|$ as in \eqref{eq:Bij} shows
    \begin{equation*}
        \frac{|N(a_{ij})|}{B_{ij}}\cdot
        \frac{|N(a_{ik})|}{B_{ik}}\cdot
        \frac{|N(a_{jk})|}{B_{jk}}< \frac{|N(a_i^2a_j^2a_k^2a_l^{-1})|}{u_\cfrb B} \ll \frac{1}{T_2^d},
    \end{equation*}
    using \eqref{eq:bound_T2}. Hence at least one factor on the left-hand side must be $\ll T_2^{-d/3}$. Therefore, without loss of generality, we may assume that $|N(a_{12})| \ll T_2^{-d/3}B_{12}$.

    Using this bound and arguing as in the proof of Lemma~\ref{lem:bound_a''}, the number of $\ab''$ is
    \begin{equation*}
        \ll \frac{W^{2d} B}{T_2^{\frac{d}{3}}|N(a_1\cdots a_4)|}.
    \end{equation*}
    After summing over $\ab'$ (where we can use \eqref{eq:sum_ai} because of \eqref{eq:bound_T2} and Lemma~\ref{lem:bound_a'_B}), this gives the satisfactory error term $W^{2d}T_2^{-d/3} B(\log B)^4$.

    Hence we can replace $\Aover$ by $A$ in the main term of Lemma~\ref{lem:restrict_a'}. 
\end{proof}

\section{Estimation of the main contribution}\label{sec:main_contribution}

\subsection{M\"obius inversion}

In the following, we remove the coprimality conditions \eqref{eq:gcdK} on $\afr_{ij}$ by M\"obius inversion. This will lead to the conditions
\begin{align}
    &\dfr_i+\afr_j=\OK,\label{eq:dfrb}\\
    &\efr_i \mid \afr_i,\label{eq:efrb}\\
    &\ffr_{ij}:=(\dfr_i\cap\dfr_j)\efr_k\efr_l \mid \afr_{ij},\label{eq:moebius}
\end{align}
with $\dfrb = (\dfr_1,\dots,\dfr_4), \efrb = (\efr_1,\dots,\efr_4) \in \IK^4$. In combination with the congruence conditions obtained from the torsor equations (as in Lemma~\ref{eq:dependent_aij}), we will obtain a lattice point counting problem in Proposition~\ref{prop:congruences_as_lattice}.

We introduce the notation
\begin{equation*}
    \mu_K(\dfrb):=\mu_K(\dfr_1)\cdots\mu_K(\dfr_4),
\end{equation*}
and define $\mu_K(\efrb)$ analogously. Let $\mu_K(\dfrb,\efrb):=\mu_K(\dfrb)\mu_K(\efrb)$.
Furthermore, let $\afrb'=(\afr_1,\afr_2,\afr_3,\afr_4)$ and encode the remaining coprimality conditions from \eqref{eq:gcdK} between them in the function
\begin{equation*}
    \theta_0(\afrb'):=\begin{cases}
        1, & \text{$\afr_i+\afr_j=\OK$,}\\
        0, & \text{else.}
    \end{cases}
\end{equation*}

\begin{lemma}\label{lem:moebius}
    Let $\ab' \in \Os_*'$ be such that $\theta_0(\afrb')=1$. Then
    \begin{equation*}
        |A(W,\ab',\cfrb,B)| = \sums{\dfrb : \eqref{eq:dfrb}\\\efrb : \eqref{eq:efrb}}
        \mu_K(\dfrb,\efrb) |\{\ab'' \in \Os'' : \eqref{eq:torsor}, \eqref{eq:in_F0W}, \eqref{eq:moebius}\}|.
    \end{equation*}
\end{lemma}

\begin{proof}
    Since $\theta_0(\afrb')=1$, the remaining coprimality conditions from \eqref{eq:gcdK} give
    \begin{align*}
        A&=A(W,\ab',\cfrb,B)=\{\ab'' \in \Os'' : \eqref{eq:torsor}, \eqref{eq:in_F0W}, \afr_i+\afr_{jk}=\afr_{ij}+\afr_{ik}=\OK\}.
    \end{align*}
    We remove all conditions $\afr_{ij}+\afr_{ik}=\OK$ by M\"obius inversion.
    Since $\afr_i+\afr_j=\OK$ and $\afr_i+\afr_{jk}=\OK$, the $i$-th torsor equation \eqref{eq:torsor} shows that a divisor $\dfr_i$ of $\afr_{ij}$ and $\afr_{ik}$ must also divide $\afr_{il}$ (as $\dfr_i \mid \afr_{ij}$ is coprime to $\afr_l$). Therefore,
    \begin{equation*}
        |A| = \sum_{\dfrb} \mu_K(\dfrb) 
        |\{\ab'' \in \Os'': \eqref{eq:torsor}, \eqref{eq:in_F0W}, 
        \afr_i+\afr_{jk}=\OK,\ \dfr_i \mid \afr_{ij},\afr_{ik},\afr_{il}\}|.
    \end{equation*}
    We may restrict the sum to \eqref{eq:dfrb} because otherwise there is no $\ab''$ satisfying all the conditions (since $\dfr_i \mid \afr_{ik}$ implies $\dfr_i+\afr_j \mid \afr_{ik}+\afr_j = \OK$).
    
    Next, we remove the conditions $\afr_i+\afr_{jk}=\OK$ by M\"obius inversion. Here, the $j$-th torsor equation shows that a divisor $\efr_i$ of $\afr_i$ and $\afr_{jk}$ must also divide $\afr_{jl}$ since $\efr_i \mid \afr_i$ is coprime to $\afr_l$. Analogously, $\efr_i \mid \afr_{kl}$. Therefore,
    \begin{align*}
        |A| &= \sum_{\dfrb:\eqref{eq:dfrb}} 
        \sum_{\efrb:\eqref{eq:efrb}} \mu_K(\dfrb,\efrb)
        |\{\ab'' \in \Os'': \eqref{eq:torsor}, \eqref{eq:in_F0W},
        \dfr_i \mid \afr_{ij},\afr_{ik},\afr_{il},\ \efr_i \mid \afr_{jk},\afr_{jl},\afr_{kl}\}|\\
        &= \sum_{\dfrb:\eqref{eq:dfrb}} 
        \sum_{\efrb:\eqref{eq:efrb}} \mu_K(\dfrb,\efrb)
        |\{\ab'' \in \Os'': \eqref{eq:torsor}, \eqref{eq:in_F0W},
        \eqref{eq:moebius}\}|,
    \end{align*}
    where we note for the last equality that $\efr_i$ is coprime to $\efr_j,\dfr_j$ (by combining $\efr_i\mid\afr_i$ and $\efr_j\mid\afr_j$ with the fact that both $\afr_j$ and $\dfr_j$ are coprime to $\afr_i$).
\end{proof}

We define the fractional ideals
\begin{align*}
    \bfr_{12} &:= (\dfr_1\cap\dfr_2\cap(\dfr_3+\dfr_4))\efr_3\efr_4\Os_{12},\\
    \bfr_{23} &:= (\dfr_2\cap\dfr_3\cap\dfr_4)\efr_1\Os_{23},\\
    \bfr_{34} &:= (\dfr_1\cap\dfr_3\cap\dfr_4)\efr_2\Os_{34}.
\end{align*}
Let
\begin{equation*}
    \Gs=\Gs(\ab',\cfrb,\dfrb,\efrb)
\end{equation*}
be the additive subgroup of $K^3$ consisting of all $(a_{12},a_{23},a_{34})$ with
\begin{equation*}
    a_{12} \in \bfr_{12},\quad
    a_{23} \in \gamma_{23}a_{12}+\afr_4\bfr_{23},\quad
    a_{34} \in \gamma_{34}a_{23}+\afr_1\bfr_{34}. 
\end{equation*}
Here, $\gamma_{23} := \gamma_{23}^*/a_3 \in K$ where $\gamma_{23}^* \in \OK$ satisfies
\begin{equation*}
    \congr{\gamma_{23}^*}{0}{\afr_3\cfr_1\dfr_3(\dfr_3+\dfr_4)^{-1}\efr_1},\quad
    \congr{\gamma_{23}^*}{a_1}{\afr_4\cfr_1\dfr_4(\dfr_3+\dfr_4)^{-1}},
\end{equation*}
and $\gamma_{34} := \gamma_{34}^*/a_4 \in K$ where $\gamma_{34}^* \in \OK$ satisfies
\begin{equation*}
    \congr{\gamma_{34}^*}{0}{\afr_4\cfr_2\dfr_4(\dfr_1+\dfr_4)^{-1}\efr_2},\quad
    \congr{\gamma_{34}^*}{a_2}{\afr_1\cfr_2\dfr_1(\dfr_1+\dfr_4)^{-1}}.
\end{equation*}
Such a $\gamma_{23}^*$ exists by the Chinese remainder theorem as 
\begin{equation*}
    \afr_3\cfr_1\dfr_3(\dfr_3+\dfr_4)^{-1}\efr_1
    +\afr_4\cfr_1\dfr_4(\dfr_3+\dfr_4)^{-1}
    = \cfr_1
\end{equation*}
(the greatest common divisor of the moduli) divides $a_1\OK=\afr_1\cfr_1$. The discussion of the existence of $\gamma_{34}^*$ is similar.

\begin{prop}\label{prop:congruences_as_lattice}
    Let $\ab' \in \Os_*'$ be such that $\theta_0(\afrb')=1$. Then
    \begin{equation*}
        |A(W,\ab',\cfrb,B)| = \sums{\dfrb : \eqref{eq:dfrb}\\\efrb : \eqref{eq:efrb}}\mu_K(\dfrb,\efrb)|\Gs(\ab',\cfrb,\dfrb,\efrb)\cap \Fs_0^{(W)}(\ab';u_\cfrb B)|.
    \end{equation*}
\end{prop}

\begin{proof}
    In view of Lemma~\ref{lem:moebius}, we are interested in
    \begin{equation*}
         A'=A'(W,\ab',\cfrb,\dfrb,\efrb;B):=\{\ab'' \in \Os'' : \eqref{eq:torsor}, \eqref{eq:in_F0W}, \eqref{eq:moebius}\}.
    \end{equation*}
    For $a_{24}$ as in \eqref{eq:dependent_aij}, the corresponding $\afr_{24}$ must be divisible by $\ffr_{24}$ \eqref{eq:moebius}. This is equivalent to
    \begin{equation}\label{eq:mod_a4d4}
        \congr{a_3a_{23}}{a_1a_{12}}{a_4\ffr_{24}\Os_{24}},
    \end{equation}
    where the modulus can be rewritten as $\afr_4\Os_4\ffr_{24}\Os_{24}=\afr_4\cfr_0\cfr_2^{-1}\ffr_{24}$. Similarly, for $\afr_{13}$,
    \begin{equation}\label{eq:mod_a1d1}
        \congr{a_4a_{34}}{a_2a_{23}}{\afr_1\cfr_0\cfr_3^{-1}\ffr_{13}},
    \end{equation}
    while \eqref{eq:moebius} for $\afr_{14}$ holds automatically under these two congruences.

    This shows that
    \begin{equation*}
        |A'|=\left|\bigwhere{(a_{12},a_{23},a_{34}) \in \Fs_0^{(W)}}
        {&a_{12} \in \Os_{12},\ \ffr_{12} \mid \afr_{12}, \\
        &a_{23} \in \Os_{23},\ \ffr_{23} \mid \afr_{23},\ \eqref{eq:mod_a4d4},\\
        &a_{34} \in \Os_{34},\ \ffr_{34} \mid \afr_{34},\ \eqref{eq:mod_a1d1}}\right|.
    \end{equation*}

    We can rewrite $a_{34} \in \Os_{34},\ \ffr_{34} \mid \afr_{34}$ as $\congr{a_{34}}{0}{\ffr_{34}\Os_{34}}$, and multiplying by $a_4$ gives $\congr{a_4a_{34}}{0}{\ffr_{34}\afr_4\cfr_0\cfr_3^{-1}}$.
    By the Chinese remainder theorem for not necessarily coprime moduli, the two congruence conditions on $a_4a_{34}$ are compatible if and only if their greatest common divisor
    \begin{equation*}
        \afr_1\cfr_0\cfr_3^{-1}\ffr_{13}+\afr_4\cfr_0\cfr_3^{-1}\ffr_{34} = \cfr_0\cfr_3^{-1}(\dfr_3 \cap (\dfr_1+\dfr_4))\efr_1\efr_2\efr_4
    \end{equation*}
    divides $a_2a_{23}\OK=\afr_2\afr_{23}\cfr_0\cfr_3^{-1}$, which is equivalent to $\dfr_1+\dfr_4 \mid \afr_{23}$ (using that $\efr_1\efr_2\efr_4$ divides $\afr_2\afr_{23}$ and is coprime to $\dfr_3\cap(\dfr_1+\dfr_4)$, and $\dfr_3 \mid \afr_{23}$). Under this condition, $a_4a_{34}$ is unique modulo their least common multiples 
    \begin{equation*}
        \afr_1\cfr_0\cfr_3^{-1}\ffr_{13}\cap \afr_4\cfr_0\cfr_3^{-1}\ffr_{34} = \afr_1\afr_4\cfr_0\cfr_3^{-1}(\dfr_1\cap\dfr_3\cap\dfr_4)\efr_2 = a_4\afr_1\bfr_{34}.
    \end{equation*}
    The conditions $\dfr_1+\dfr_4 \mid \afr_{23}$ and $\ffr_{23} \mid \afr_{23}$ can be combined to $(\dfr_2\cap\dfr_3\cap(\dfr_1+\dfr_4))\efr_1\efr_4 \mid \afr_{23}$. Hence $|A'|$ equals
    \begin{equation*}
        \left|\bigwhere{(a_{12},a_{23},a_{34}) \in \Fs_0^{(W)}}
        {
            &a_{12} \in \Os_{12}, (\dfr_1\cap\dfr_2)\efr_3\efr_4 \mid \afr_{12}, \\
            &a_{23} \in \Os_{23}, (\dfr_2\cap\dfr_3\cap(\dfr_1+\dfr_4))\efr_1\efr_4 \mid \afr_{23}, \eqref{eq:mod_a4d4},\\
            &\congr{a_4a_{34}}{0}{\afr_4\cfr_0\cfr_3^{-1}\ffr_{34}}, \eqref{eq:mod_a1d1}
        }\right|.
    \end{equation*}

    As for $a_{34}$, we rewrite the first two conditions on $a_{23}$ as
    \begin{equation*}
        \congr{a_3a_{23}}{0}{\afr_3\cfr_0\cfr_2^{-1}(\dfr_2\cap\dfr_3\cap(\dfr_1+\dfr_4))\efr_1\efr_4}.
    \end{equation*}
    This is compatible with \eqref{eq:mod_a4d4} if their greatest common divisor $\efr_1\efr_3\efr_4(\dfr_2\cap(\dfr_1+\dfr_4)\cap(\dfr_3+\dfr_4))\cfr_0\cfr_2^{-1}$ divides $\afr_1\afr_{12}$, which reduces to $\dfr_3+\dfr_4\mid \afr_{12}$. This combines with the other divisibility condition on $\afr_{12}$ to $(\dfr_1\cap\dfr_2\cap(\dfr_3+\dfr_4))\efr_3\efr_4 \mid \afr_{12}$, which is equivalent to $a_{12} \in \bfr_{12}$. Under this condition, $a_3a_{23}$ is unique modulo the least common multiple
    \begin{equation*}
        \afr_3\afr_4\cfr_0\cfr_2^{-1}(\dfr_2\cap\dfr_3\cap\dfr_4)\efr_1 = a_3\afr_4\bfr_{23}.
    \end{equation*}
    Hence 
    \begin{equation*}
        |A'|=
        \left|\left\{
            \begin{aligned}
                &(a_{12},a_{23},a_{34}) \in \Fs_0^{(W)} : a_{12} \in \bfr_{12}, \\
                &\congr{a_3a_{23}}{0}{\afr_3\cfr_0\cfr_2^{-1}(\dfr_2\cap\dfr_3\cap(\dfr_1+\dfr_4))\efr_1\efr_4},\ \eqref{eq:mod_a4d4},\\
                &\congr{a_4a_{34}}{0}{\afr_4\cfr_0\cfr_3^{-1}\ffr_{34}},\ \eqref{eq:mod_a1d1}
            \end{aligned}
        \right\}\right|.
    \end{equation*}

    Finally, we check that all $(a_{12},a_{23},a_{34}) \in \Gs$ actually satisfy the conditions in this description of $A'$.
\end{proof}

\subsection{The case of large M\"obius variables}

In order to deal with the error terms in the following steps of our proof, we must restrict the range for $\dfrb,\efrb$ in Proposition~\ref{prop:congruences_as_lattice}. This should be compared to \cite[Proposition~4]{Bre02}.

For this restriction (in Proposition~\ref{prop:large_moebius} below), we will need the following upper bound for the number of points in the fundamental domain on the torsor that satisfy the height conditions but not necessarily the coprimality conditions \eqref{eq:gcdK}:

\begin{lemma}\label{lem:upperbound}
    We have
    \begin{equation*}
        |\{(\ab',\ab'') \in (\Os' \times \Os'') \cap \Fs : \text{\eqref{eq:heightK}, \eqref{eq:torsor}}
        \}| \ll B(\log B)^6.
    \end{equation*}
\end{lemma}

\begin{proof}
As in Lemma~\ref{lem:weyl_group_symmetry}, it suffices to bound the number of solutions satisfying the symmetry condition \eqref{eq:symmetryK}.

As in Proposition~\ref{prop:E_WK}, we can associate with each solution and each $v\mid\infty$ numbers $W_v \ge 1$, and with each $i$ a set $V_i \subset \Omega_\infty$, numbers $W_i \ge 1$, and a parameter $W_0=\max_{v\mid\infty} W_v^{1/d_v}$. To satisfactorily bound the contribution from solutions with $\max_i W_i \ge 1$, it suffices to estimate the number of solutions with $W_2=\max_i W_i$ as $\ll W_0^{-2}B(\log B)^6$ for any $W_0 \ge 1$. Note that these arguments did not use the coprimality conditions and hence can be transferred verbatim.

When counting the solutions, we can now see from Lemma \ref{lem:dependent_aij} that $a_{23}$ is in general only uniquely determined modulo $\frac{\afr_4}{\afr_3+\afr_4} \Os_{23}$, and $a_{34}$ is determined uniquely modulo $\frac{\afr_1}{\afr_1+\afr_4} \Os_{34}$. We can thus bound the number of solutions by
\[\sum_{a_{12}} \sums{a_{23}\\ \Mod{\frac{\afr_4}{\afr_3+\afr_4} \Os_{23}}} \sums{a_{34}\\ \Mod{\frac{\afr_1}{\afr_1+\afr_4} \Os_{34}}} 1,\]
where the sum is restricted to $|a_{ij}|_v \ll W_v^{-2}|B_{ijv}|_v$ for $v \in V_i,V_j$ and to $|a_{ij}|_v \ll W_v|B_{ijv}|_v$ for $v \in V_k,V_{l}$. The only difference compared to the proof of Proposition~\ref{prop:E_WK} is the introduction of the greatest common divisors here.

Following the same lines of computations as in that proof, we find that the number of choices for $a_{12}$ is still bounded by $\ll W_1^{-2}W_2^{-2}W_3W_4 B_{12}$, while the number of choices for $a_{23}$ is now bounded by 
\begin{equation*}
    \ll \frac{W_1}{W_2W_3} \cdot \frac{B_{23} \N(\afr_3+\afr_4)}{\N\afr_4},
\end{equation*}
and similarly the number of choices for $a_{34}$ is bounded by
\begin{equation*}
    \ll \frac{W_1W_2}{W_3W_4} \cdot \frac{B_{34}\N(\afr_1+\afr_4)}{\N\afr_1}.
\end{equation*}
The total contribution then becomes 
\[\ll \frac{1}{W_0^2} \cdot \frac{B}{\N(\afr_1\afr_2\afr_3\afr_4)} \cdot \N((\afr_1+\afr_4)(\afr_3+\afr_4)).\]
Summing this over $\N\afr_1,\dots,\N\afr_4 \ll B$ (since \eqref{eq:heightK} implies \eqref{eq:bound_ai_B}), we can bound this contribution by
\begin{equation*}
    \ll \frac{B}{W_0^2} \sum_{\afr_4} \sum_{\dfr, \efr \mid \afr_4} \sums{\afr_1,\afr_2,\afr_3:\\ \afr_1=\dfr \afr_1'\\\afr_3=\efr \afr_3'} \frac{\N(\dfr\efr)}{\N(\afr_2\afr_4\dfr \afr_1'\efr \afr_3')} 
    \ll \frac{B}{W_0^2}\sum_{\afr_4} \sum_{\dfr, \efr \mid \afr_4} \frac{(\log B)^3}{\N\afr_4} 
    \ll \frac{B(\log B)^6}{W_0^2}
\end{equation*}
by several applications of Lemma~\ref{lem:sum_log}.

Therefore, we have a satisfactory bound for the contribution of all solutions except for those satisfying $\vert x_{ijv}\vert_v \le |B_{ijv}|_v$ (i.e., condition \eqref{eq:bound_SFv} with $W=1$).

But here we can repeat the argument from the proof of Lemma \ref{lem:bound_a''} and obtain the bound
\[\ll \frac{B}{\N(\afr_1\afr_2\afr_3\afr_4)} \cdot \N((\afr_1+\afr_4)(\afr_3+\afr_4)),\]
which (by the same computation as above) yields a satisfactory contribution $\ll B(\log B)^6$ after summing over the $\afr_i$.
\end{proof}

From here, we use the parameter $T_1 = \exp(c_1 \log B / \log \log B)$ as in Section~\ref{sec:notation}.

\begin{prop}\label{prop:large_moebius}
    We have
    \begin{align*}
        |\Mover^{(\id)}_\cfrb(B)| &=  \sums{\ab' \in \Os_*' \cap \Fs_1^4\\\eqref{eq:bound_T2}} \theta_0(\afrb') 
        \sums{\dfrb: \eqref{eq:dfrb},\ \N\dfr_i\le T_1\\\efrb:\eqref{eq:efrb},\ \N\efr_i \le T_1}\mu_K(\dfrb,\efrb) |\Gs(\cfrb,\ab',\dfrb,\efrb)\cap \Fs_0^{(W)}(\ab';u_\cfrb B)|\\
        &+ O\left(\frac{B(\log B)^4}{(\log \log B)^{\frac{1}{3d+1}}}\right).
    \end{align*}
\end{prop}

\begin{proof}
    Our starting point is Proposition~\ref{prop:remove_symmetry} combined with Proposition~\ref{prop:congruences_as_lattice}.
    We need to discard tuples $(\mathbf{a}',\mathbf{a}'',\dfrb,\efrb)$ with $\max_i\{\N\dfr_i,\N\efr_i\}>T_1$ satisfying \eqref{eq:dfrb}, \eqref{eq:efrb} and \eqref{eq:moebius}. We begin with some preliminary manoeuvring.

    In the first step, we replace the ideals $\dfrb$ by ideals $\dfrb'$ that are pairwise coprime (note that the $\efr_i$ are already pairwise coprime and also coprime to the $\dfr_i$), but still satisfy the analogue of \eqref{eq:moebius}. We can do this by considering their prime factorization and choosing for each prime ideal dividing $\dfr_1\dfr_2\dfr_3\dfr_4$ an index $i$ such that it divides $\dfr_i$ to maximal power. We then let $\dfr_i'$ to be the product of prime powers associated with $i$ in this way. Clearly the ideals $\dfr_i'$ are pairwise coprime, and if one of the $\dfr_i$ has norm at least $T_1$, then $\N(\dfr_1'\dfr_2'\dfr_3'\dfr_4') \ge T_1$. Moreover, since $\dfr_i \mid \dfr_1'\dfr_2'\dfr_3'\dfr_4'$, the number of $\dfr_i$ associated with the same $\dfr_i'$ is bounded by a divisor function, which in turn is bounded by $T_1^{1/4}$ if we choose the constant $c_1$ sufficiently large (see \eqref{eq:divisor_bound}). It therefore suffices to prove that the number of tuples $(\mathbf{a}',\mathbf{a}'',\dfrb',\efrb)$ with $\max_i\{\N(\dfr_1'\dfr_2'\dfr_3'\dfr_4'),\N(\efr_i)\}>T_1$ is $O(T_1^{-1/2}B(\log B)^6)$.

    In the second step, we pass to principal ideals. Indeed, we can find principal ideals $(d_i)$ and $(e_i)$ such that $(d_i)\dfr_i'^{-1}, (e_i)\efr_i^{-1}$ are integral ideals of norm $O(1)$. Since the number of $(d_i), (e_i)$ associated with the same $\dfr_i'$ and $\efr_i$ is $O(1)$, it suffices to count tuples $(\mathbf{a}', \mathbf{a}'', \db, \eb)$ with
    \begin{equation*}
        \max_i\{|N(d_1d_2d_3d_4)|,|N(e_i)|\} \gg T_1.
    \end{equation*}

    Moreover, after multiplying our solution $(\mathbf{a}',\mathbf{a}'')$ by a suitable constant (e.g., the fourth power of the product of the possible norms of $(d_i)\dfr_i'^{-1}, (e_i)\efr_i^{-1}$), we can make sure that the respective variables remain integral in the following computation, i.e., they are actually divisible by the $d_i,e_i$, respectively.

    We are now ready to initiate the main argument: Acting as in \eqref{eq:action} with
    \begin{equation}\label{eq:rescaling}
    \left(\frac{1}{e_1e_2e_3e_4}, \frac{d_1}{e_1}, \frac{d_2}{e_2}, \frac{d_3}{e_3}, \frac{d_4}{e_4}\right) \in (K^\times)^5
    \end{equation}
    on the tuples $(\mathbf{a}',\mathbf{a}'')$ gives a new solution of \eqref{eq:torsor} of height
    \begin{equation*}
        \ll \frac{B}{|N(d_1d_2d_3d_4e_1^2e_2^2e_3^2e_4^2)|}.
    \end{equation*}
    Indeed, this replaces $a_i$ by $a_id_ie_i^{-1}$ and $a_{ij}$ by $a_{ij}(d_id_je_ke_{l})^{-1}$, which remain integral by construction (in particular using that the $d_i$ are pairwise coprime by the first step of our proof). For the height function, we note that the polynomials $\Pt(a_{1},\dots,a_{34})$ (of anticanonical degree) are multiplied by
    \[\left(\frac{1}{e_1e_2e_3e_4}\right)^3\left(\frac{d_1}{e_1} \cdot \frac{d_2}{e_2} \cdot \frac{d_3}{e_3} \cdot \frac{d_4}{e_4}\right)^{-1}=\frac{1}{d_1d_2d_3d_4e_1^2e_2^2e_3^2e_4^2},\]
    leading to the desired bound.

    While the new solution does not necessarily lie in $\Fs$, we observe that the action of \eqref{eq:rescaling} commutes with the action of $U_K \times (\OK^\times)^4$, so that we obtain an induced action on the set of orbits under $U_K \times (\OK^\times)^4$. Each of these orbits contains a unique element of $\Fs$ counted in Lemma~\ref{lem:upperbound}. Hence the number of such solutions is
    \begin{equation*}
        \ll \frac{B(\log B)^6}{|N(d_1d_2d_3d_4e_1^2e_2^2e_3^2e_4^2)|}.
    \end{equation*}
    Moreover, note that we can reconstruct the $d_i$ up to a divisor function from this new solution.

    Hence the number of old solutions with $|N(d_1d_2d_3d_4)| \gg T_1$ is bounded by
    \[\sum_{(e_i)} \frac{B(\log B)^6}{T_1^{\frac 1 2}} \cdot \frac{1}{|N(e_1^2e_2^2e_3^2e_4^2)|} \ll \frac{B(\log B)^6}{T_1^{\frac 1 2}},\]
    while the number of old solutions with $|N(e_1)| \gg T_1$ is bounded by
    \[\sums{|N(e_1)| \gg T_1\\(e_2),(e_3),(e_{4})} \frac{B(\log B)^6T_1^{\frac 1 2}}{|N(e_1^2e_2^2e_3^2e_4^2)|} \ll \frac{B(\log B)^6}{T_1^{\frac 1 2}}\]
    as desired, again if we choose $c_1$ in the definition of $T_1$ sufficiently large to majorize the divisor bound \eqref{eq:divisor_bound} by $T_1^{1/2}$. The contribution from $|N(e_i)| \gg T_1$ satisfies the same bound by symmetry.
\end{proof}

\subsection{Counting via o-minimal structures}

By Proposition~\ref{prop:congruences_as_lattice}, we must estimate the number of lattice points in the set $\Fs_0^{(W)}$. In order to control the difference to the expected main term, we observe that $\Fs_0^{(W)}$ is defined in Wilkie's o-minimal structure $\RR_{\exp}$ \cite{Wilkie96} (with the exponential function appearing in the construction of our fundamental domain in Section~\ref{sec:fundamental_domain}). Hence we can apply the adaptation of the Lipschitz principle to the framework of o-minimal structures by Barroero and Widmer \cite{BW14}, as in \cite[\S8--10]{FP16}.

Let $\ab' \in \Os_*'$ with $\theta_0(\afrb')=1$, and $\dfrb,\efrb \in \IK^4$ with \eqref{eq:dfrb}, \eqref{eq:efrb}. 
Using the $\RR$-linear isomorphism
\begin{equation}\label{eq:def_tau}
    \tau : \prod_{v \mid \infty} K_v^3 \to \prod_{v \mid \infty} K_v^3, \quad  (x_{12v},x_{23v},x_{34v})_v\mapsto \left(\frac{x_{12v}}{\N\bfr_{12}^{\frac 1 d}},\frac{x_{23v}}{\N(\afr_4\bfr_{23})^{\frac 1 d}},\frac{x_{34v}}{\N(\afr_1\bfr_{34})^{\frac 1 d}}\right)_v,
\end{equation}
we define $\Lambda=\Lambda(\ab',\cfrb,\dfrb,\efrb):=\tau(\sigma(\Gs(\ab',\cfrb,\dfrb,\efrb)))$. We have
\begin{equation}\label{eq:lattices}
    |\Gs \cap \Fs_0^{(W)}| = |\sigma(\Gs)\cap S_F^{(W)}| = |\Lambda \cap \tau(S_F^{(W)})|.
\end{equation}

For every coordinate subspace $S$ of $\prod_{v \mid \infty} K_v^3 = \RR^{3d}$, let $V_S = V_S(\ab',\cfrb,\dfrb,\efrb,u_\cfrb B)$ be the $(\dim S)$-dimensional volume of the orthogonal projection of $\tau(S_F^{(W)})$ to $S$.

\begin{lemma}\label{lem:o-minimal_estimation}
    We have
    \begin{equation*}
        |\Gs(\ab',\cfrb,\dfrb,\efrb) \cap \Fs_0^{(W)}(\ab';u_\cfrb B)| = \frac{2^{3r_2}\vol S_F^{(W)}(\ab';u_\cfrb B)}{|\Delta_K|^{\frac 3 2}\N(\afr_1\afr_4\bfr_{12}\bfr_{23}\bfr_{34})} + O\left(\sum_S V_S\right),
    \end{equation*}
    where $S$ runs through all proper coordinate subspaces of $\RR^{3d}$.
\end{lemma}

\begin{proof}
    As in \cite[Lemma~8.2]{FP16} and \cite[Lemma~6.5]{DP20}, we show that $\Lambda$ is a lattice of rank $3d$ and determinant $(2^{-r_2}|\Delta_K|^{1/2})^3$, with first successive minimum $\lambda_1 \ge 1$ (using that either $a_{12} \in \bfr_{12}$ is nonzero, or $a_{23} \in \afr_4\bfr_{23}$ is nonzero, or $a_{34} \in \afr_1\bfr_{34}$ is nonzero).

    Let $Z$ be the set of
    \begin{equation*}
        (\beta,\beta',\beta_{12},\beta_{23},\beta_{34},(x_{1v},\dots,x_{4v},x_{12v},x_{23v},x_{34v})_v) \in \RR^5 \times \prod_{v\mid\infty} K_v^7
    \end{equation*}
    such that, with the notation
    \begin{align*}
        z_{ijv} &:= \beta_{ij}x_{ijv}\text{ for $(i,j)=(1,2),(2,3),(3,4)$},\\
        z_{13v} &:= \frac{x_{2v}z_{23v}-x_{4v}z_{34v}}{x_{1v}},\\ 
        z_{24v} &:= \frac{x_{3v}z_{23v}-x_{1v}z_{12v}}{x_{4v}},\\
        z_{14v} &:= \frac{x_{2v}x_{3v}z_{23v}-x_{3v}x_{4v}z_{34v}-x_{1v}x_{2v}z_{12v}}{x_{1v}x_{4v}}
    \end{align*}
    (see \eqref{eq:dependent_aij}, \eqref{eq:dependent_xijv}), we have
    \begin{align*}
        &\beta,\beta',\beta_{12},\beta_{23},\beta_{34} > 0,\\
        &|x_{iv}|_v > 0\text{ and }0 < |z_{ijv}|_v \le \left|\beta'\frac{(\beta \prod_{w \mid \infty}|x_{1w}x_{2w}x_{3w}x_{4w}|_w)^{\frac 1 {3d}}}{x_{iv}x_{jv}}\right|_v \text{ for all $v \mid \infty$},\\
        &(\Nt_v(x_{1v},\dots,x_{4v},z_{12v},\dots,z_{34v})^{\frac 1 3})_v \in \exp(F(\beta^{\frac{1}{3d}})),
    \end{align*}
    with the coordinate-wise exponential function $\exp$. Let $Z_T \subset \prod_{v \mid \infty} K_v^3$ be the fiber of $Z$ over $T=(\beta,\beta',\beta_{12},\beta_{23},\beta_{34},(x_{1v},\dots,x_{4v})_v)$. Then $\tau(S_F^{(W)}))$ is $Z_T$ for
    \begin{equation*}
        T := (u_\cfrb B, W, \N\bfr_{12}^{\frac 1 d}, \N(\afr_4\bfr_{23})^{\frac 1 d}, \N(\afr_1\bfr_{34})^{\frac 1 d}, (a_1^{(v)},\dots,a_4^{(v)})_v).
    \end{equation*}
    
    As in \cite[\S 9]{FP16}, we observe that $Z$ is definable in $\RR_{\exp}$, and all fibers $Z_T$ are bounded. Therefore, an application of \cite[Theorem~1.3]{BW14} to the right-hand side of \eqref{eq:lattices} gives the result.
\end{proof}

To estimate the $V_S$, we recall that every $(x_{12v},x_{23v},x_{34v})_v \in S_F^{(W)}$ satisfies \eqref{eq:bound_SFv} by definition.

Let $\tau_v: K_v^3 \to K_v^3$ be the $v$-component of $\tau$ as in \eqref{eq:def_tau}. Let
\begin{equation*}
    S_F^{(W,v)}:=\{(x_{12v},x_{23v},x_{34v}) \in K_v^3 : \text{\eqref{eq:bound_SFv} for all $i,j$}\},
\end{equation*}
hence
\begin{equation}\label{eq:tau_v}
    \tau(S_F^{(W)}) \subset \prod_{v\mid \infty} \tau_v(S_F^{(W,v)}).
\end{equation}

\begin{lemma}\label{lem:volume_projections}
    Let $v \mid \infty$. For $P_v = (p_{12},p_{23},p_{34}) \in \{0,\dots, d_v\}^3$, let $V_{P_v}$ be the volume of the orthogonal projection of $\tau_v(S_F^{(W,v)})$ to
    \begin{equation*}
        \RR^{3d_v-p_{12}-p_{23}-p_{34}} \cong \{\text{$f_{ij}(x_{ij})=0$ for all $(i,j) \in \{(1,2),(2,3),(3,4)\}$}\} \subset \RR^{3d_v},
    \end{equation*}
    where
    \begin{equation*}
        f_{ij}(x) = \begin{cases}
            0, & \text{$p_{ij}=0$,}\\
            x, & \text{$p_{ij}=1$, $v$ real,}\\
            \text{$\Re x$ or $\Im x$}, & \text{$p_{ij}=1$, $v$ complex,}\\
            x, & \text{$p_{ij}=2$, $v$ complex.}
        \end{cases}
    \end{equation*}
    Then
    \begin{equation*}
        V_{P_v} \ll \left(\frac{W^{3d}B}{|N(a_1\cdots a_4)|}\right)^{\frac{d_v}{d}}\left(W^dB_{12}\right)^{-\frac{p_{12}}{d}}
        \left(\frac{W^dB_{23}}{|N(a_4)|}\right)^{-\frac{p_{23}}{d}}
        \left(\frac{W^dB_{34}}{|N(a_1)|}\right)^{-\frac{p_{34}}{d}}.
    \end{equation*}
\end{lemma}

\begin{proof}
    We note that $S_F^{(W,v)}$ is a product of real intervals of length bounded by \eqref{eq:bound_SFv} or of complex balls of radius bounded by \eqref{eq:bound_SFv}. Hence (using \eqref{eq:ai_conjugates} to obtain $|B_{ijv}|_v \asymp B_{ij}^{d_v/d}$, for example)
    \begin{align*}
        V_{P_v} &\ll \left(\frac{W^dB_{12}}{\N\bfr_{12}}\right)^{\frac{d_v-p_{12}}{d}}
        \left(\frac{W^dB_{23}}{\N(\afr_4\bfr_{23})}\right)^{\frac{d_v-p_{23}}{d}}
        \left(\frac{W^dB_{34}}{\N(\afr_1\bfr_{34})}\right)^{\frac{d_v-p_{34}}{d}}\\
        &\ll  \left(W^dB_{12}\right)^{\frac{d_v-p_{12}}{d}}
        \left(\frac{W^dB_{23}}{|N(a_4)|}\right)^{\frac{d_v-p_{23}}{d}}
        \left(\frac{W^dB_{34}}{|N(a_1)|}\right)^{\frac{d_v-p_{34}}{d}},
    \end{align*}
    where the denominators come from the construction of $\tau_v$. We obtain the result using the identity $B_{12}B_{23}B_{34} = u_\cfrb B/|N(a_2a_3)|$.
\end{proof}

We define
\begin{equation}\label{eq:theta_T1}
    \theta(\afrb',T_1):=\sums{\dfrb: \eqref{eq:dfrb},\ \N\dfr_i\le T_1\\\efrb:\eqref{eq:efrb},\ \N\efr_i \le T_1}\frac{\mu_K(\dfrb,\efrb)}{\N(\bfr_{12}\Os_{12}^{-1}\bfr_{23}\Os_{23}^{-1}\bfr_{34}\Os_{34}^{-1})}.
\end{equation}

\begin{prop}\label{prop:errors_volumes}
    We have
    \begin{equation*}
        |\Mover_\cfrb^{(\id)}(B)| = \sums{\ab' \in \Os_*' \cap \Fs_1^4\\\eqref{eq:bound_T2}} 
        \frac{2^{3r_2}\theta_0(\afrb')\theta(\afrb',T_1) \vol S_F^{(W)}(\ab';u_\cfrb B)}{|\Delta_K|^{\frac 3 2}\N(\afr_1\afr_4\Os_{12}\Os_{23}\Os_{34})}
        + O\left(\frac{B(\log B)^4}{(\log \log B)^{\frac{1}{3d+1}}}\right).
    \end{equation*}
\end{prop}

\begin{proof}
    We plug Lemma~\ref{lem:o-minimal_estimation} into Proposition~\ref{prop:large_moebius}. For every proper coordinate subspace $S$ of $\RR^{3d}$, it remains to prove that
    \begin{equation*}
        \sums{\ab' \in \Os_*' \cap \Fs_1^4\\\eqref{eq:bound_T2}} \theta_0(\afrb') \sums{\dfrb: \eqref{eq:dfrb},\ \N\dfr_i\le T_1\\\efrb:\eqref{eq:efrb},\ \N\efr_i\le T_1}|\mu_K(\dfrb,\efrb)| V_S \ll \frac{B(\log B)^4}{(\log \log B)^{\frac{1}{3d+1}}}.
    \end{equation*}
    
    Indeed, by \eqref{eq:tau_v}, $V_S \le \prod_{v \mid \infty} V_{P_v}$ for certain $P_v$ that are not all $(0,0,0)$ since $S \subsetneq \RR^{3d}$. Multiplying the bounds from Lemma~\ref{lem:volume_projections} gives
    \begin{equation*}
        V_S \ll \left(\frac{W^{3d}B}{|N(a_1\cdots a_4)|}\right)\left(W^dB_{12}\right)^{-e_{12}}
        \left(\frac{W^dB_{23}}{|N(a_4)|}\right)^{-e_{23}}
        \left(\frac{W^dB_{34}}{|N(a_1)|}\right)^{-e_{34}},
    \end{equation*}
    where at least one of the nonnegative $e_{12},e_{23},e_{34}$ is at least $1/d$.
    By \eqref{eq:bound_T2}, we have $|N(a_i^2a_j^2a_k^2a_l^{-1})| \ll T_2^{-d}u_\cfrb B$, which is equivalent to $|N(a_i)| \ll T_2^{-d/3}B_{jk}$. Hence
    \begin{equation*}
        \left(\frac{W^dB_{jk}}{|N(a_i)|}\right)^{-\frac{1}{d}} \ll \frac{1}{WT_2^{\frac 1 3}} \ll \frac{1}{T_2^{\frac 1 3}}.
    \end{equation*}
    In total, $V_S \ll W^{3d}T_2^{-1/3}B/|N(a_1\cdots a_4)|$. The summation over $\dfrb,\efrb$ gives a factor $T_1^8$, and for the summation over $\ab'$, we use Lemma~\ref{lem:bound_a'_B} and \eqref{eq:sum_ai} to obtain a total bound of $O(W^{3d}T_1^8T_2^{-1/3}B(\log B)^4)$. This is satisfactory by our choice of parameters $W,T_1,T_2$ (specifically since $c_2>24c_1$).
\end{proof}

\subsection{The archimedean densities}

We compute the volume of $S_F^{(W)}$ by comparing it to $S_F$. Here, the archimedean densities appear, as computed in Lemma~\ref{lem:expected_archimedean_densities}.

We write $\oneb = (1,1,1,1) \in \RR^4$.

\begin{lemma}\label{lem:vol_SFW}
    For $\ab' \in \Os_*'$, we have
    \begin{equation*}
        \vol(S_F^{(W)}(\ab';u_\cfrb B)) = \frac{u_\cfrb B}{|N(a_2a_3)|} \cdot \vol(S_F^{(W)}(\oneb;1)).
    \end{equation*}
\end{lemma}

\begin{proof}
    We transform $S_F^{(W)}$ using
    \begin{equation*}
        z_{ijv} := \frac{x_{ijv}}{B_{ijv}}
    \end{equation*}
    for $(i,j)=(1,2),(2,3),(3,4)$. By \eqref{eq:Bij}, its Jacobian determinant is
    \begin{equation*}
        \prod_{v\mid\infty} B_{12v}B_{23v}B_{34v} = \frac{u_\cfrb B}{|N(a_2a_3)|}.
    \end{equation*}
    Using $a_j^{(v)}B_{ijv}=a_k^{(v)}B_{ikv}$, we see that the expressions from \eqref{eq:dependent_xijv} lead to
    \begin{equation*}
        z_{13v}:=z_{23v}-z_{34v} = \frac{a_2^{(v)}B_{23v}z_{23v}-a_4^{(v)}B_{34v}z_{34v}}{a_1^{(v)}B_{13v}} = \frac{x_{13v}}{B_{13v}},
    \end{equation*}
    and similarly
    \begin{equation*}
        z_{24v}:=z_{23v}-z_{12v}=\frac{x_{24v}}{B_{24v}},\quad z_{14v}:=z_{23v}-z_{34v}-z_{12v}=\frac{x_{14v}}{B_{14v}}.
    \end{equation*}
    Hence the anticanonical polynomials in the height function transform as
    \begin{equation*}
        \Pt(a_1^{(v)},\dots,a_4^{(v)},x_{12v},\dots,x_{34v})
        =\frac{(u_\cfrb B|N(a_1\cdots a_4)|)^{\frac{1}{d}}}{\sigma_v(a_1\cdots a_4)}\Pt(1,1,1,1,z_{12v},\dots,z_{34v}),
    \end{equation*}
    and therefore
    \begin{equation*}
        \Nt_v(\ab';x_{12v},x_{23v},x_{34v})=\frac{(u_\cfrb B|N(a_1\cdots a_4)|)^{\frac{d_v}{d}}}{|a_1\cdots a_4|_v} \Nt_v(\oneb;z_{12v},z_{23v},z_{34v}),
    \end{equation*}
    so that the height condition \eqref{eq:heightK} turns into $\prod_{v \mid \infty} \Nt_v(\oneb;z_{12v},z_{23v},z_{34v}) \le 1$.
\end{proof}

\begin{lemma}\label{lem:removeW}
    We have
     \begin{equation*}
        \vol(S_F^{(W)}(\oneb;1)) = \vol(S_F(\oneb;1))+O(W^{-2}).
    \end{equation*}
\end{lemma}

\begin{proof}
    This is similar to the proof of Proposition~\ref{prop:E_WK} (but easier since we have integrals instead of sums now, which also results in a slightly better bound).  We need to estimate $\vol(S_F^{(W)}(\oneb;1)\setminus S_F(\oneb;1))$. As in the proof of Proposition~\ref{prop:E_WK}, to each $(z_{ijv})_v$ and each $v \mid \infty$ we can attach some $W_v \ge 1$ and some index $i$ such that $|z_{ijv}|_v, |z_{ikv}|_v,|z_{il v}|_v \ll W_v^{-2}$ and $|z_{jkv}|_v, |z_{kl v}|_v, |z_{jl v}|_v \le W_v$. We also define $V_i$ to be the set of $v \mid \infty$ with this index $i$ and put $W_i=\prod_{v \in V_i} W_v \ge 1$.
    
    By symmetry and a dyadic decomposition of the ranges of the $W_v$, it then again suffices to prove that for a fixed choice of $(W_v)_{v\mid\infty}$ with $W_0=\max_{v\mid\infty} W_v^{1/d_v}$ and $W_2=\max_i W_i$, the volume of such $(z_{ijv})_v$ is $\ll W_0^{-3}$ for any $W_0 \ge 1$.

    The variable $x_{ijv}$ is now restricted to a region of volume $\ll W_i^{-2}W_j^{-2}W_kW_l$. Hence the total volume of $(x_{12v},x_{23v}, x_{34v})_v$ is bounded by $\ll W_2^{-3}W_3^{-3} \le W_0^{-3}$, as desired.
\end{proof}

\begin{lemma}\label{lem:volume_S_F}
  With $\omega_v(X)$ as in our Theorem, we have
  \begin{equation*}
    \vol(S_F(\oneb;1)) = \frac{1}{3}\cdot 2^{r_1}\cdot \left(\frac{\pi}{4}\right)^{r_2}\cdot \left(\prod_{v\mid \infty} \omega_v(X)\right)\cdot R_K.
  \end{equation*}
\end{lemma}

\begin{proof}
    In view of \eqref{eq:lines}, we make the change of variables $z_{12v}=y_3$, $z_{23v}=y_1$, $z_{34v}=y_1-y_2$, hence (using the notation from Lemma~\ref{lem:vol_SFW}) $z_{13v}=z_{23v}-z_{34v}=y_2$, $z_{24v}=z_{23v}-z_{12v}=y_1-y_3$, $z_{14v}=z_{23v}-z_{34v}-z_{12v}=y_2-y_3$ (with Jacobian determinant $1$). This transforms $\Pt(1,1,1,1,z_{12v},\dots,z_{34v})$ appearing  in $S_F(\oneb;1)$ into the polynomials $P(y_1,y_2,y_3)$, for all $\Pt$ corresponding to $P \in \Ps$ as in \eqref{eq:def_Ptilde}.

    Defining
    \begin{equation*}
        N_v(y_1,y_2,y_3) := \max_{P \in \Ps}|P(y_1,y_2,y_3)|_v,
    \end{equation*}
    this shows that
    \begin{equation*}
        \vol(S_F(\oneb;1)) = \int_{\frac 1 3 (\log N_v(y_{1v},y_{2v},y_{3v}))_{v \mid \infty} \in F(1)} \prod_{v \mid \infty} \ddd y_{1v} \ddd y_{2v} \ddd y_{3v}.
    \end{equation*}
    From here, we have the same computation as in \cite[Lemma~5.1]{FP16}.
\end{proof}

\begin{lemma}\label{lem:archimedean_density_under_symmetry}
    For $s \in S$ and $v \mid \infty$, let $\omega_v^{(s)}(X)$ be defined as $\omega_v(X)$ in our Theorem, but with $\Ps$ replaced by $\Ps^{(s)}$. Then $\omega_v^{(s)}(X) = \omega_v(X)$.
\end{lemma}

\begin{proof}
    By definition~\eqref{eq:def_Ptilde} and since $P$ is homogeneous of degree $3$, we have
    \begin{equation*}
        |\Pt(\ab',\ab'')|_v = \left|P\left(\frac{a_{23}|a_1\cdots a_4|_v^{\frac 2 {3d_v}}}{a_1a_4},\frac{a_{13}|a_1\cdots a_4|_v^{\frac 2 {3d_v}}}{a_2a_4},\frac{a_{12}|a_1\cdots a_4|_v^{\frac 2 {3d_v}}}{a_3a_4}\right)\right|_v.
    \end{equation*}
    
    For $s=s_4$ and $y = (y_1,y_2,y_3) \in (K_v^\times)^3$, we compute (with $a_{ij}=y_k$ and $a_{il}=y_j-y_k$ for $j<k$ and $l=4$ in the first step)
    \begin{align*}
        |P^{(s)}(y)|_v
        &= |\Pt^{(s)}(1,1,1,1,y_3,y_2,y_2-y_3,y_1,y_1-y_3,y_1-y_2)|_v\\
        &= |\Pt(y_1,y_2,y_3,-1,1,1,y_2-y_3,1,y_1-y_3,y_1-y_2)|_v\\
        &= \left|P\left(\frac{|y_1y_2y_3|_v^{\frac 2 {3d_v}}}{y_1},\frac{|y_1y_2y_3|_v^{\frac 2 {3d_v}}}{y_2},\frac{|y_1y_2y_3|_v^{\frac 2 {3d_v}}}{y_3}\right)\right|_v.
    \end{align*}
    Therefore,
    \begin{align*}
        \omega_v^{(s)}(X) &= \vol\{y \in K_v^3 : \max_{P^{(s)} \in \Ps^{(s)}} |P^{(s)}(y)|_v \le 1\}\\
        &=\vol\{z \in K_v^3 : \max_{P \in \Ps} |P(z)|_v \le 1\} = \omega_v(X),
    \end{align*}
    using the change of coordinates $z_i=|y_1y_2y_3|_v^{2/(3d_v)}/y_i$ for $i \in \{1,2,3\}$, whose Jacobian has absolute value $1$.

    The transformations in the cases $s \in \{s_1,s_2,s_3\}$ are similar and differ from the one above only by a linear change of coordinates, e.g., $(y_1,y_2,y_3) \mapsto (y_1-y_3,y_2-y_3,y_3)$ in the case of $s_3$. The case $s=\id$ is trivial.
\end{proof}

\subsection{The nonarchimedean densities}

Recall the definition of $\theta(\afrb',T_1)$ in \eqref{eq:theta_T1}. We remove the restrictions on $\dfrb,\efrb$ introduced in Proposition~\ref{prop:large_moebius} and compute its Euler product.

\begin{lemma}\label{lem:euler_product}
    For $\afrb' \in \IK^4$ with $\theta_0(\afrb')=1$, we have
    \begin{equation*}
        \theta(\afrb',T_1) = \theta(\afrb')+O(T_1^{-\frac 1 2}),
    \end{equation*}
    where 
    \[\theta(\afrb')=\prod_{\pfr \mid \afr_1\afr_2\afr_3\afr_4} \left(1-\frac{1}{\N\pfr}\right)\left(1-\frac{1}{\N\pfr^2}\right) \cdot \prod_{\pfr \nmid \afr_1\afr_2\afr_3\afr_4} \left(1-\frac{4}{\N\pfr^2}+\frac{3}{\N\pfr^3}\right).\]
    The total contribution of the error term to $|\Mover^{(\id)}_\cfrb(B)|$ is $O(T_1^{-1/2}B(\log B)^4)$, which is sufficient.
\end{lemma}

\begin{proof}
    We begin by noting that
    \begin{align*}
    \theta(\afrb',T_1) 
        &=\!\!\!\sums{\dfrb: \eqref{eq:dfrb},\ \N\dfr_i\le T_1\\\efrb:\eqref{eq:efrb},\ \N\efr_i \le T_1} \frac{\mu_K(\dfrb,\efrb)}{\N((\dfr_1\cap\dfr_2\cap(\dfr_3+\dfr_4))(\dfr_2\cap\dfr_3\cap\dfr_4)(\dfr_1\cap\dfr_3\cap\dfr_4)\efr_1\efr_2\efr_3\efr_4)}\\
        &=D(\ab',T_1) \prod_{i=1}^4 E(\afr_i, T_1),
    \end{align*}
    where
    \[D(\ab',T_1):=\sums{\dfrb: \eqref{eq:dfrb},\ \N\dfr_i\le T_1} \frac{\mu_K(\dfrb)}{\N((\dfr_1\cap\dfr_2\cap(\dfr_3+\dfr_4))(\dfr_2\cap\dfr_3\cap\dfr_4)(\dfr_1\cap\dfr_3\cap\dfr_4))}\]
    and
    \[E(\afr, T_1):=\sum_{\efr \mid \afr,\ \N\efr \le T_1} \frac{\mu_K(\efr)}{\N\efr}.\]
    
    If we let
    \[E(\afr):=\sum_{\efr \mid \afr} \frac{\mu_K(\efr)}{\N\efr}=\prod_{\pfr \mid \afr} \left(1-\frac{1}{\N\pfr}\right),\]
    then clearly $E(\afr) \in (0,1]$ and
    \[\vert E(\afr)-E(\afr,T_1)\vert \le \sum_{\efr \mid \afr,\ \N\efr >T_1} \frac{1}{\N\efr} \le \frac{\tau_K(\afr)}{T_1} \ll \frac{1}{T_1^{\frac 1 2}},\]
    where we use \eqref{eq:divisor_bound} to bound the divisor function by $T_1^{1/2}$ upon choosing $c_1$ in the definition of $T_1$ sufficiently large.

    Similarly, if we let
    \[D(\ab'):=\sums{\dfrb: \eqref{eq:dfrb}} \frac{\mu_K(\dfrb)}{\N((\dfr_1\cap\dfr_2\cap(\dfr_3+\dfr_4))(\dfr_2\cap\dfr_3\cap\dfr_4)(\dfr_1\cap\dfr_3\cap\dfr_4))},\]
    then we can compute $D(\ab')$ as an Euler product since everything is multiplicative. If $\pfr$ does not divide any of the $\afr_i$, then the Euler factor is $1-4\N\pfr^{-2}+3\N\pfr^{-3}$ since the numerator will be $1$ if none of the $\dfr_i$ are divisible by $\pfr$, $\N\pfr^2$ in the four cases where exactly one of the $\dfr_i$ is divisible by $\pfr$, and $\N\pfr^3$ in the remaining $11$ cases (seven of which give a positive sign, and four a negative). In particular, the summand is symmetric in $\dfr_1,\dots,\dfr_4$. On the other hand, if $\pfr$ divides one of the $\afr_i$, then the Euler factor is $1-\N\pfr^{-2}$ since only $\dfr_i$ can be divisible by $\pfr$, which contributes $-\N\pfr^{-2}$.

    Thus, we have
    \[D(\ab')=\prod_{\pfr \mid \afr_1\afr_2\afr_3\afr_4} \left(1-\frac{1}{\N\pfr^2}\right) \cdot \prod_{\pfr \nmid \afr_1\afr_2\afr_3\afr_4} \left(1-\frac{4}{\N\pfr^2}+\frac{3}{\N\pfr^3}\right)\]
    and, in particular, $D(\ab') \in (0,1]$.

Moreover, by symmetry and an application of Rankin's trick, we have
\begin{align*}
    &\vert D(\ab',T_1)-D(\mathbf{a})\vert\\
    &\ll \sum_{\N\dfr_1 >T_1} \sum_{\dfr_2,\dfr_3,\dfr_4} \frac{\mu_K(\dfrb)^2}{\N((\dfr_1\cap\dfr_2\cap(\dfr_3+\dfr_4))(\dfr_2\cap\dfr_3\cap\dfr_4)(\dfr_1\cap\dfr_3\cap\dfr_4))}\\
    &\le \frac{1}{T_1^c}\sum_{\dfrb} \frac{\mu_K(\dfrb)^2\N\dfr_1^c}{\N((\dfr_1\cap\dfr_2\cap(\dfr_3+\dfr_4))(\dfr_2\cap\dfr_3\cap\dfr_4)(\dfr_1\cap\dfr_3\cap\dfr_4))}\\
    &\le \frac{1}{T_1^c} \prod_{\pfr} \left(1+O(\N\pfr^{c-2})\right) \ll \frac{1}{T_1^c}
\end{align*}
for any fixed $c \in (0,1)$ by a similar computation of the Euler factors as above. The equality in the lemma now follows by collecting the results.

Finally, the error term leads to a total contribution bounded by
\[\ll T_1^{-\frac 1 2}\sums{\ab' \in \Os_*' \cap \Fs_1^4\\\eqref{eq:bound_T2}} \frac{B}{|N(a_1a_2a_3a_4)|}\ll \frac{B(\log B)^4}{T_1^{\frac 1 2}},\]
as desired.
\end{proof}

\subsection{Completion of the proof} 

We collect the previous results and perform the remaining summations over the elements $\ab'$ after transforming them into summations over ideals. Here, we apply results from \cite{DF14}. Finally, we remove the conditions \eqref{eq:bound_T2} and recover the factor $\alpha(X)$ of Peyre's constant.

\begin{prop}\label{prop:summary_ideal_summation}
    We have
    \begin{equation*}
       \sum_{\cfrb \in \Cs} |\Mover_\cfrb^{(\id)}(B)| = \frac{2^{3r_2}\vol(S_F^{(W)}(\oneb;1))h_K}{|\Delta_K|^{\frac 3 2}} \!\!\!\sums{\afrb' \in \IK^4\\\eqref{eq:bound_T2}} \frac{\theta_0(\afrb')\theta(\afrb')B}{\N(\afr_1\afr_2\afr_3\afr_4)} + O\left(\frac{B(\log B)^4}{(\log \log B)^{\frac{1}{3d+1}}}\right).
    \end{equation*}
\end{prop}

\begin{proof}
    By Proposition~\ref{prop:errors_volumes}, Lemma~\ref{lem:vol_SFW}, and Lemma~\ref{lem:euler_product}, we have
    \begin{align*}
        \sum_{\cfrb \in \Cs} |\Mover_\cfrb^{(\id)}(B)| &=
        \frac{2^{3r_2}\vol(S_F^{(W)}(\oneb;1))}{|\Delta_K|^{\frac 3 2}} \sum_{\cfrb \in \Cs} \sums{\ab' \in \Os_*'\cap \Fs_1^4\\\eqref{eq:bound_T2}} \frac{\theta_0(\afrb')\theta(\afrb')u_\cfrb B}{|N(a_2a_3)|\N(\afr_1\afr_4\Os_{12}\Os_{23}\Os_{34})}\\ &+ O(B(\log B)^4 (\log \log B)^{\frac{-1}{3d+1}}).
    \end{align*}
    By definition, $|N(a_2a_3)|\N(\Os_{12}\Os_{23}\Os_{34}) = u_\cfrb\cdot \N(\afr_2\afr_3)$. Hence our main term is
    \begin{equation*}
        \frac{2^{3r_2}\vol(S_F^{(W)}(\oneb;1))}{|\Delta_K|^{\frac 3 2}} \sum_{\cfrb \in \Cs} \sums{\ab' \in \Os_*'\cap \Fs_1^4\\\eqref{eq:bound_T2}} \frac{\theta_0(\afrb')\theta(\afrb')B}{\N(\afr_1\afr_2\afr_3\afr_4)}.
    \end{equation*}
    We observe that $\Os_1,\dots,\Os_4$ are independent of $\cfr_0$, which gives a factor $h_K$. 
    Note that $a_i \in \Os_{i*} \cap \Fs_1$ is equivalent to $\afr_i = a_i\Os_i^{-1}=a_i\cfr_i^{-1}$, running through all nonzero ideals in the same class as $\cfr_i^{-1}$, which runs through representative of the ideal class group. Hence we obtain a sum over $\afrb' = (\afr_1,\dots,\afr_4) \in \IK^4$.
\end{proof}

\begin{lemma}\label{lem:remain_sum}
    We have
    \begin{equation*}
        \sums{\afrb' \in \IK^4\\\eqref{eq:bound_T2}} \frac{\theta_0(\afrb')\theta(\afrb')B}{\N(\afr_1\afr_2\afr_3\afr_4)} 
        = \frac{3\alpha(X)}{5} \rho_K^4 \theta_1 B(\log B)^4+O\left(\frac{B(\log B)^4}{\log \log B}\right),
    \end{equation*}
    where
    \begin{equation*}
        \theta_1 = \prod_{\pfr}
        \left(1-\frac{1}{\N\pfr}\right)^5
        \left(1+\frac{5}{\N\pfr}+\frac{1}{\N\pfr^2}\right).
    \end{equation*}
\end{lemma}

\begin{proof}
    As in \cite[\S 12]{FP16}, we apply \cite[Proposition~7.2]{DF14} inductively to
    \begin{equation*}
        V(t_1,\dots,t_4;B):=\frac{B}{t_1t_2t_3t_4}\cdot V'(t_1,\dots,t_4;B),
    \end{equation*}
    where $V'$ is the indicator function of the set of all $t_1,\dots,t_4 \ge 1$ satisfying
    \begin{equation}\label{eq:bound_ti_T2}
        t_i^2t_j^2t_k^2t_l^{-1} \le  \frac{B}{T_2^d}
    \end{equation}
    (see \eqref{eq:bound_T2}). This gives
    \begin{equation*}
        \sums{\afrb' \in \IK^4\\\eqref{eq:bound_T2}} \frac{\theta_0(\afrb')\theta(\afrb')B}{\N(\afr_1\afr_2\afr_3\afr_4)} = \rho_K^4 \theta_1 V_0(B)+O(B(\log B)^3(\log \log B)),
    \end{equation*}
    where $\theta_1$ is the ``iterated average'' of $\theta_0(\afrb')\theta(\afrb')$ (see \cite[\S 2]{DF14}; here clearly $\theta_0(\afrb')\theta(\afrb') \in \Theta_4'(4)$), which has the value given above by \cite[Lemma~2.8]{DF14}, and
    \begin{equation*}
        V_0(B):=\int_{\substack{t_1,\dots,t_4 \ge 1\\\eqref{eq:bound_ti_T2}}} \frac{B}{t_1t_2t_3t_4} \ddd t_1 \cdots \ddd t_4.
    \end{equation*}
    With the definition
    \begin{equation*}
        V_1:=\vol\{(x_1,x_2,x_3,x_4) \in \RR_{\ge 0}^4 : 2x_i+2x_j+2x_k-x_l \le 1\},
    \end{equation*}
    substituting $t_i = (B/T_2^d)^{x_i}$ and using $\log T_2 \ll \log B/\log \log B$ shows that
    \begin{equation*}
        V_0(B) = V_1\cdot B(\log(B/T_2^d))^4 
        = V_1\cdot B(\log B)^4+O\left(\frac{B(\log B)^4}{\log \log B}\right),
    \end{equation*}    
    where the error term is sufficiently small.
    
    By \cite[(3.25)]{Bre02} and \eqref{eq:alpha}, $V_1 = 1/180 = 3\alpha(X)/5$.
\end{proof}

\begin{proof}[Proof of the Theorem (see Section~\ref{sec:main_result})]
    Proposition~\ref{prop:summary_ideal_summation} with Lemma~\ref{lem:remain_sum} shows that $\sum_{\cfrb \in \Cs}|\Mover_\cfrb^{(\id)}(B)|$ is
    \begin{equation*}
        \frac{3\alpha(X)}{5} \frac{2^{3r_2}\vol(S_F^{(W)}(\oneb;1))h_K}{|\Delta_K|^{\frac 3 2}}\rho_K^4\theta_1 B(\log B)^4 +O\left(\frac{B(\log B)^4}{(\log \log B)^{\frac{1}{3d+1}}}\right).
    \end{equation*}
    By Lemma~\ref{lem:removeW}, replacing $S_F^{(W)}$ by $S_F$ gives an error term $O(W^{-2}B(\log B)^4)$, which is satisfactory. Using Lemma~\ref{lem:volume_S_F} for $\vol(S_F)$ gives 
    \begin{equation*}
        \sum_{\cfrb \in \Cs}|\Mover_\cfrb^{(\id)}(B)| = \frac{\alpha(X)}{5} \frac{|\mu_K|}{|\Delta_K|}\rho_K^5\theta_1\left(\prod_{v\mid\infty} \omega_v(X)\right) B(\log B)^4 +O\left(\frac{B(\log B)^4}{(\log \log B)^{\frac{1}{3d+1}}}\right).
    \end{equation*}
    
    Let $s \in S$. We note that this estimation holds for any choice of representatives $\Cs$ of $\Cl_K^5$ (in particular for $\Cs^{(s)}$ as defined in \eqref{eq:def_C^s}) and for any choice of $\Ps$ that fulfils the assumptions in Section~\eqref{sec:heights} (in particular for $\Ps^{(s)}$ as defined in \eqref{eq:def_P^s}). Therefore, we get the same estimation for
    \begin{equation*}
        \sum_{\cfrb \in \Cs} |\Mover_{s(\cfrb)}^{(s)}(B)| = \sum_{\cfrb \in \Cs^{(s)}} |\Mover_{\cfrb}^{(s)}(B)|,
    \end{equation*}
    except that $\Ps$ is replaced by $\Ps^{(s)}$ in the archimedean densities $\omega_v(X)$. But Lemma~\ref{lem:archimedean_density_under_symmetry} shows that this replacement gives the same densities $\omega_v^{(s)}(X) = \omega_v(X)$. Hence the summation over $s \in S$ in Lemma~\ref{lem:weyl_group_symmetry} gives a factor $5$.
    
    Furthermore, clearly everything here and before works with \eqref{eq:symmetryK} replaced by \eqref{eq:symmetry_strict}, with exactly the same result once we have removed these symmetry conditions (see Proposition~\ref{prop:remove_symmetry}). Therefore, for all $s \in S$, we obtain the same estimation for $\sum_{\cfrb \in \Cs}|\Munder_{s(\cfrb)}^{(s)}(B)|$ as for $\sum_{\cfrb \in \Cs} |\Mover_{s(\cfrb)}^{(s)}(B)|$. Plugging this into Proposition~\ref{prop:torsor_parameterization}, Lemma~\ref{lem:weyl_group_symmetry}, and \eqref{eq:orbit_to_fundamental_domain} gives the same estimation of $N_{U,H}(B)$ from below and from above, which completes the proof of the main theorem.
\end{proof}

\bibliographystyle{alpha}

\bibliography{manin_dp5nf}

\end{document}